\documentclass[12pt, oneside, a4paper, reqno]{amsart} 

\usepackage{latexsym}
\usepackage{amssymb}
\usepackage{amsmath}
\usepackage{mathrsfs}
\usepackage{bbold}

\usepackage[marginparsep=0pt, left=2.5cm, right=2.5cm, top=3cm, bottom=3cm]{geometry}
\usepackage{setspace}
\onehalfspace

\usepackage{amsthm}
\theoremstyle{plain}
\newtheorem{X}{X}[section]
\newtheorem{Lem}[X]{Lemma}
\newtheorem{Thm}[X]{Theorem}
\newtheorem{Prop}[X]{Proposition}

\usepackage{hyperref}
\hypersetup{
    pdftoolbar=true,                        
    pdfmenubar=false,                        
    pdffitwindow=false,                      
    pdfstartview={FitH},                    
    pdftitle={Strings on congruent primes in short intervals II},    
    pdfauthor={Tristan Freiberg},    
    pdfcreator={Tristan Freiberg},
    pdfproducer={Tristan Freiberg},
    pdfnewwindow=true,                       
    colorlinks=true,
    linkcolor={black},                      
    citecolor={black},
    filecolor={black},
    urlcolor={black},         
}
\usepackage{cite}

\numberwithin{equation}{section}

\renewcommand{\le}{\ensuremath{\leqslant}}
\renewcommand{\ge}{\ensuremath{\geqslant}}

\newcommand{\br}[1]{\ensuremath{\left(#1\right)}} 
\newcommand{\Br}[1]{\ensuremath{\left\{#1\right\}}}   

\newcommand{\ab}[1]{\ensuremath{\vert#1\vert}} 
\newcommand{\abs}[1]{\ensuremath{\left\lvert#1\right\rvert}} 

\newcommand{\sums}[2]{\ensuremath{\sum_{\substack{#1 \\ #2}}}}
\newcommand{\sumss}[3]{\ensuremath{\sum_{\substack{#1 \\ #2 \\ #3}}}}

\newcommand{\prods}[2]{\ensuremath{\prod_{\substack{#1 \\ #2}}}}

\newcommand{\BigO}[1]{\ensuremath{O\br{#1}}} 
\newcommand{\Oh}{\ensuremath{O}} 
\newcommand{\littleo}[1]{\ensuremath{o\br{#1}}} 

\newcommand{\defeq}{\ensuremath{:=}} 
\newcommand{\eqdef}{\ensuremath{=:}} 

\newcommand{\li}[1]{\ensuremath{\mathrm{li}\br{#1}}} 

\newcommand{\dd}[1]{\ensuremath{\mathrm{d}#1}}

\title{Strings of congruent primes in short intervals II}

\author{Tristan Freiberg}
\address{Institutionen f\"or matematik, KTH, 100 44 Stockholm, Sweden}
\email{tristanf@kth.se}

\thanks{The author is supported by the G\"oran Gustafsson Foundation (KVA)}

\begin{document}

\begin{abstract}
Let $p_1 = 2, p_2 = 3,\ldots$ be the sequence of all primes. Let $\epsilon$ be an arbitrarily small but fixed positive number, and fix a coprime pair of integers $q \ge 3$ and $a$. We will establish a lower bound for the number of primes $p_r$, up to $X$, such that both $p_{r+1} - p_{r} < \epsilon \log p_r$ and $p_{r} \equiv p_{r+1} \equiv a \bmod q$ simultaneously hold. As a lower bound for the number of primes satisfying the latter condition, the bound we obtain improves upon a bound obtained by D.~Shiu.
\end{abstract}

\maketitle

\section{Introduction}\label{Section 1: Introduction}

Let $p_1 = 2, p_2 = 3,\ldots$ be the sequence of all primes, and let $\epsilon$ be an arbitrarily small but fixed positive number. In 2005 \cite{GPY2005, GPY2009}, Goldston, Pintz, and Y{\i}ld{\i}r{\i}m made a significant breakthrough by proving that $p_{r+1} - p_r < \epsilon\log p_r$ for infinitely many pairs $p_r, p_{r+1}$ of primes. That is, for infinitely many $r$, the $r$th prime gap, $p_{r+1} - p_{r}$, is arbitrarily small compared to the `expected' gap of $\log p_r$. In 2006 \cite{GPY2006} they extended their method to prove an analogous result for primes in arithmetic progressions. Thus, given a coprime pair of integers $q \ge 3$ and $a$, if $p'_1 < p'_2 < \cdots$ is the sequence of all primes congruent to $a \bmod q$, then for infinitely many pairs $p'_{m}, p'_{m+1}$, we have $p'_{m+1} - p'_{m} < \epsilon\log p_m$. 

Given any such pair $p'_{m}, p'_{m+1}$, there may or may not be a third prime $p$, not congruent to $a \bmod q$, such that $p'_{m} < p < p'_{m+1}$. Thus, \emph{either} there are infinitely many triples of primes $p_{r},p_{r+1},p_{r+2}$, not necessarily in the same arithmetic progression mod $q$, such that $p_{r+2} - p_r < \epsilon\log p_r$; \emph{or} there are infinitely many pairs of consecutive primes $p_{r}, p_{r+1}$ such that both $p_{r+1} - p_{r} < \epsilon \log p_r$ and $p_{r} \equiv p_{r+1} \equiv a \bmod q$ simultaneously hold. Presumably both statements are true, but one can only deduce that one of them is true, and one does not know which one, from the result in \cite{GPY2006}.

Although we would like to prove the first statement, unfortunately it seems beyond reach of the method of Goldston, Pintz, and Y{\i}ld{\i}r{\i}m, at least at present. (See \cite[\S 1, Question 3]{GPY2009}.) It is natural, then, to ask whether one can at least prove the second statement. In so doing, one would establish a conjecture of Chowla that there are infinitely many pairs of consecutive primes $p_{r}, p_{r+1}$ such that $p_{r} \equiv p_{r+1} \equiv a \bmod q$. This conjecture was in fact already proved by D.~Shiu in 2000 \cite{S2000}. 

As it turns out, the ideas of Shiu can be combined with those of Goldston, Pintz, and Y{\i}ld{\i}r{\i}m to prove that there are indeed infinitely many pairs of consecutive primes $p_{r}, p_{r+1}$ such that both $p_{r+1} - p_{r} < \epsilon \log p_r$ and $p_{r} \equiv p_{r+1} \equiv a \bmod q$ simultaneously hold. We did this in \cite{F2010}, where we also obtained a very weak quantitative result \cite[\S 7]{F2010}: there is a positive constant $A = A(q)$, depending only on $q$, such that for all sufficiently large $X$, 
\begin{align}\label{(1.1)}
 \sumss{p_{r} \le X}{p_{r+1} - p_{r} < \epsilon \log p_{r}}{p_{r} \equiv p_{r+1} \equiv a \bmod q}
\hspace{-5pt} 1 \ge X^{1/3(\log\log X)^A}.
\end{align}
Our purpose here is to improve this lower bound to the following:
\begin{Thm}\label{Theorem 1.1}
 Let $p_1 = 2, p_2 = 3,\ldots$ be the sequence of all primes. Fix any positive number $\epsilon$, and fix a pair of coprime integers $q \ge 3$ and $a$. There is an absolute positive constant $c$ such that, for all sufficiently large $X$,
\begin{align}\label{(1.2)}
 \sumss{p_{r} \le X}{p_{r+1} - p_{r} < \epsilon \log p_{r}}{p_{r} \equiv p_{r+1} \equiv a \bmod q}
\hspace{-5pt} 1 \ge X^{1 - c/\log\log X}.
\end{align}
\end{Thm}
As a lower bound for the number of primes $p_{r}$ up to $X$ for which $p_{r} \equiv p_{r+1} \equiv a \bmod q$, \eqref{(1.2)} is, once $X$ is sufficiently large, greater than that obtained by Shiu \cite[Theorem 2]{S2000}, namely $X^{1 - \varepsilon(X)}$, where
\begin{align*}
\varepsilon(X) = C_1(q)\br{\frac{\log\log\log X}{\log\log X}}^{1/\phi(q)}
\end{align*}
if $a \equiv \pm 1 \bmod q$, and
\begin{align*}
\varepsilon(X) = C_2(q)\br{\frac{(\log\log\log X)^2}{(\log\log X)(\log\log\log\log X)}}^{1/\phi(q)}
\end{align*}
otherwise. (Here, $C_1(q)$ and $C_2(q)$ are constants depending only on $q$.)

\section{Discussion}\label{Section 2: Discussion}

The way to incorporate the ideas of Shiu into the work of Goldston, Pintz, and Y{\i}ld{\i}r{\i}m is explained in \cite[\S 2]{F2010}. Basically, Goldston, Pintz, and Y{\i}ld{\i}r{\i}m \cite{GPY2006} proved that for all sufficiently large $N$, there is at least one integer\footnote{In fact, a lower bound for the number of such integers $n$, of the form $N/(\log N)^c$, $c$ a positive constant, is implicit in the work of Goldston, Pintz, and Y{\i}ld{\i}r{\i}m.} $n \in (N,2N]$ such that there are at least two primes of the form $Qn+h$, where: $Q$ is a multiple of $q$ such that $\log QN \sim \log N$; $h$ is in the set 
\begin{align*}
S = S(H) \defeq \{1 \le h \le H : \text{$(Q,h) = 1$ and $h \equiv a \bmod q$}\};
\end{align*}
and $H = \epsilon\log N$. Goldston, Pintz, and Y{\i}ld{\i}r{\i}m \cite[(2.1) --- (2.4)]{GPY2006} took\footnote{Actually, if there happens to be an exceptional modulus $q_0 \le N^{1/(\log\log N)^2}$, and if $p_0$ is its greatest prime factor, we remove $p_0$ from the product defining $Q$, so that $(Q,p_0) = 1$. See \cite[Lemma 2]{GPY2006} and \cite[\S 5]{F2010} for details. We overlook this technical complication for the purposes of simplifying the present discussion.}  
\begin{align*}
Q = Q(H) \defeq q\prod_{p \in \mathscr{P}} p,
\end{align*}
where 
\begin{align*}
\mathscr{P} = \mathscr{P}(H) \defeq \{p \le H/(\log H)^2\},
\end{align*}
but if we remove from $\mathscr{P}$ any subset of the primes in the interval $(\log H, H/(\log H)^2]$, the key estimates \cite[Proposition 3.2]{F2010} still hold, with one exception --- namely, we do not necessarily have \cite[(2.2)]{F2010}: $\ab{S} \gg_q H\phi(Q)/Q$.

Our goal is to remove primes from $\mathscr{P}$ in such a way that we have the following for the resulting $Q$: almost all of the integers $h \in [1,H]$ that are coprime with $Q$ are congruent to $a \bmod q$, in the sense that if 
\begin{align*}
T = T(H) \defeq \{1 \le h \le H : \text{$(Q,h) = 1$ and $h \not\equiv a \bmod q$}\}, 
\end{align*}
then $\ab{T} = \littleo{\ab{S}}$ as $H \to \infty$; and $\ab{S} \gg_q H\phi(Q)/Q$ for all sufficiently large $H$. Since $Qn+h$ is prime only if $(Q,h) = 1$, we could deduce from this that, for infinitely many of those $n$ for which $(Qn,Qn+H]$ contains at least two primes congruent to $a \bmod q$, among those primes is a pair of consecutive primes. Indeed, we would be able to establish \eqref{(1.2)}. (See \cite[\S 4, \S 7]{F2010} for details.)

Based on a construction used by Shiu in \cite{S2000}, we defined such a set $\mathscr{P}$ in \cite[\S 6.2]{F2010} (also see \eqref{(3.1)} --- \eqref{(3.4)} below). In fact, denoting by $\mathscr{P}' = \mathscr{P}'(H)$ the set considered by Shiu, we have $\mathscr{P} = \mathscr{P}' \cup \{p \le \log H : p \equiv 1 \bmod q\}$. Since $\mathscr{P}'$ is defined in such a way that it consists only of primes up to $H/(\log H)^2$, and contains all primes $p \not\equiv 1 \bmod q$ up to $\log H$, $\mathscr{P}$ consists only of primes up to $H/(\log H)^2$ and, in particular\footnote{The fact that $\mathscr{P}$ contains all primes up to $\log H$ is used to show that, for a given $k$-tuple of linear forms $\mathcal{H} = \{Qx+h_1,\ldots Qx+h_k\}$, $h_i \in [1,H]$, $(Q,h_1\cdots h_k) = 1$, we have $\mathfrak{S}(\mathcal{H}) \sim (Q/\phi(Q))^k$ as $H \to \infty$, where $\mathfrak{S}(\mathcal{H})$ is the singular series for $\mathcal{H}$. See \cite[Lemma 5.1]{F2010} for details.}, all primes up to $\log H$.

However, in \cite{F2010}, we were only able to establish the following: $\ab{T} \ll H/\log H$ for all sufficiently large $H$; $H/\log H = \littleo{H\phi(Q)/Q}$; and there is a positive constant $A$, depending only on $q$, such that for all sufficiently large $Y$, there is some $H \in [Y/(\log Y)^{A},Y]$ for which $\ab{S} \gg_q H\phi(Q)/Q$. From this we deduced \eqref{(1.1)} in \cite[\S 7]{F2010}.

The reason we were not able to establish that $\ab{S} \gg_q H\phi(Q)/Q$ for \emph{all} sufficiently large $H$ in \cite{F2010} is that we used \cite[Lemma 2]{S2000} (Lemma 6.2 in \cite{F2010}): an asymptotic for the number of integers up to $H$ that are composed only of primes congruent to $1 \bmod q$. Defining $Q' = Q'(H)$ and $S' = S'(H)$ analogously to $Q$ and $S$, but with $\mathscr{P}'$ in place of $\mathscr{P}$, Shiu used this asymptotic to show that $\ab{S'} \gg_q H\phi(Q')/Q'$ for all sufficiently large $H$. In \cite[\S 6]{F2010}, we took this as our starting point, and then dealt with the extra primes $\{p \le \log H : p \equiv 1 \bmod q\}$ in $\mathscr{P}$.

What we need is an asymptotic for the number of integers, up to $H$, that are composed only of primes both congruent to $1 \bmod q$ \emph{and} greater than $\log H$. Much of this note is devoted to establishing such a result (Lemma \ref{Lemma 3.3} below). Using this we are able to show that $\ab{S} \gg_q H\phi(Q)/Q$ for all sufficiently large $H$. Indeed, using Lemma \ref{Lemma 3.4} (below) instead of \cite[Lemma 6.5]{F2010} in \cite[\S 7]{F2010}, we are able to establish Theorem \ref{Theorem 1.1}.

We will show that the inequalities in Lemma \ref{Lemma 3.4} hold for $q$ in a certain range depending on $H$. This uniformity is not needed to prove Theorem \ref{Theorem 1.1}, but it can be used to prove a version of Theorem \ref{Theorem 1.1} in which $q$ is allowed to tend very slowly to infinity with $X$. It is hoped to publish an account of this, in which we will also consider `strings' of more than $2$ congruent primes --- in longer intervals.

\section{Proof of Theorem 1.1}\label{Section 2: Proof of Theorem 1.1}

Throughout this section, at each and every occurrence of $\Oh$ and $\ll$, the implied constant is absolute. The letter $c$, by itself, always denotes an absolute positive constant, possibly a different constant at each occurrence. 

Theorem \ref{Theorem 1.1} will follow from Lemma \ref{Lemma 3.4}, below. Lemma \ref{Lemma 3.4} is a corollary of: Theorem \ref{Theorem 3.1}, which is a version of the Siegel-Walfisz theorem; Lemma \ref{Lemma 3.2}, which is a version of Mertens' theorem in which the primes are restricted to the arithmetic progression $1 \bmod q$; and Lemma \ref{Lemma 3.3}, which gives an asymptotic for the number of integers, up to $X$, composed only of primes that are both congruent to $1 \bmod q$ and greater than a power of $\log X$. 

In each of the lemmas below, the estimates are shown to hold uniformly for $q$ in a certain range. We do not need this uniformity to prove Theorem \ref{Theorem 1.1} --- it would suffice to use the prime number theorem for arithmetic progressions instead of Theorem \ref{Theorem 3.1}, and versions of Lemmas \ref{Lemma 3.2}, \ref{Lemma 3.3}, and \ref{Lemma 3.4} in which $q$ is arbitrary but bounded.  

We use the Siegel-Walfisz theorem, in the following form, in the proofs of Lemmas \ref{Lemma 3.2} and \ref{Lemma 3.4}:
\begin{Thm}[Siegel-Walfisz]\label{Theorem 3.1}
Fix a positive number $A$. For all sufficiently large $X$ we have, uniformly for integers $q$ satisfying $1 \le q \le (\log X)^{A}$, the following estimate:
\begin{align*}
\sums{p \le X}{p \equiv 1 \bmod q} 1 = \br{1 + \BigO{\frac{1}{\log X}}} \frac{X}{\phi(q)\log X}.
\end{align*}
\end{Thm}
\begin{proof}
Indeed, we have \cite[\S 11.3, Corollary 11.20]{MV2007}:
\begin{align*}
\sums{p \le X}{p \equiv a \bmod q} 1 = \frac{\li{X}}{\phi(q)} + \BigO{X\exp\br{-C_A\sqrt{\log X}}},
\end{align*}
uniformly for $1 \le q \le (\log X)^A$ and integers $a$ coprime with $q$, where $C_A$ is a positive constant depending only on $A$. The less precise and less general statement of Theorem \ref{Theorem 3.1}, which follows since $\li{X} = X/\log X + \BigO{X/(\log X)^2}$, is sufficient for our purposes.
\end{proof}

We will use the following version of Mertens' theorem in the proof of Lemma \ref{Lemma 3.4}:
\begin{Lem}\label{Lemma 3.2}
Fix a positive number $A$. For all sufficiently large $X$ we have, uniformly for integers $q$ satisfying $1 \le q \le (\log X)^{A}$, the following estimate:
\begin{align*}
 \prods{p \le X}{p \equiv 1 \bmod q} \br{1 - \frac{1}{p}}^{-1} = 
 \br{1 + \BigO{\frac{1}{\log X}}}  e^{\gamma/\phi(q)} c(q)(\log X)^{\frac{1}{\phi(q)}},
\end{align*}
where $\gamma = 0.57721\ldots$ is the Euler-Mascheroni constant, and $c(q)$ is a positive constant depending only on $q$. We have $c(1) = 1$ and $c(2) = 1/2$. 
\end{Lem}
\begin{proof}
The case $q = 1$ is Mertens' theorem, and the case $q = 2$ follows at once from this. We prove the result for $3 \le q \le (\log X)^A$ in \S \ref{Section 3}, where $c(q)$ is given explicitly.
\end{proof}

The following result, which reduces to \cite[Lemma 2]{S2000} in the case $Y = 1$ (and $q$ fixed), is the key that allows us to establish the inequalities in Lemma \ref{Lemma 3.4} for \emph{all} sufficiently large $H$, rather than just for a certain sequence of $H$ tending to infinity as in \cite[\S 6]{F2010}.
\begin{Lem}\label{Lemma 3.3}
Fix a positive number $A$ and a number $\alpha \in (0,\frac{1}{2})$. For all sufficiently large $X$ we have, uniformly for $Y$ satisfying $1 \le Y \le (\log X)^A$ and integers $q$ satisfying $3 \le q \le (\log X)^{\alpha}$, the following estimate: 
\begin{align*}
\sumss{n \le X}{p \mid n \Rightarrow p \equiv 1 \bmod q}{\text{and $p > Y$}} \hspace{-5pt} 1 & = 
\br{1 + \BigO{  \frac{  (\log\log X)^{c} } { {\phantom{^{1-}}} (\log X)^{1 - 2\alpha} }  } } \frac{c(q)}{\Gamma(1/\phi(q))} \cdot
\frac{X(\log X)^{\frac{1}{\phi(q)}}}{\log X} \hspace{-5pt} \prods{p \le Y}{p \equiv 1 \bmod q}\hspace{-5pt}\br{1 - \frac{1}{p}},
\end{align*}
where $c(q)$ is the positive constant, depending only on $q$, in the statement of Lemma \ref{Lemma 3.2}. 
\end{Lem}
\begin{proof}
See \S \ref{Section 3}.
\end{proof}

Before stating Lemma \ref{Lemma 3.4}, we need some definitions. Let a sufficiently large number $H$, and a coprime pair of integers $q \ge 3$ and $a$, be given. If $a \equiv 1 \bmod q$, let
\begin{align}\label{(3.1)}
 \mathscr{P}(H) \defeq \{p \le \log H : p \equiv 1 \bmod q\} \cup \{p \le H/(\log H)^2 : p \not\equiv 1 \bmod q \}. 
\end{align}
If $a \not\equiv 1 \bmod q$, define
\begin{align}\label{(3.2)}
 t(H) \defeq \exp\br{\frac{(\log H)(\log\log\log H)}{2\log\log H}},
\end{align}
and, noting that $\log H < t(H) < H/t(H) < H/(\log H)^2$ for all sufficiently large $H$, let
\begin{align}\label{(3.3)}
\begin{split}
 \mathscr{P}(H) & \defeq 
\{p \le \log H : p \equiv 1 \bmod q\} \\ & \phantom{\defeq } \cup  
\{p \le H/(\log H)^2 : \text{$p \not\equiv 1 \bmod q$ and $p \not\equiv a \bmod q$}\} \\ & \phantom{\defeq  } \cup 
\{t(H) < p \le H/(\log H)^2 : p \equiv 1 \bmod q \} \\ & \phantom{\defeq  } \cup
\{p \le H/t(H) : p \equiv a \bmod q\}.
\end{split}
\end{align}
In other words, $\mathscr{P}(H)$ consists of all primes up to $H/(\log H)^2$, except for the primes
\begin{align*}
\{ \log H < p \le t(H) : p \equiv 1 \bmod q \} \cup \{H/t(H) < p \le H/(\log H)^2 : p \equiv a \bmod q \}.
\end{align*}
In either case, set
\begin{align}\label{(3.4)}
\tilde{Q} = \tilde{Q}(H;q,a) \defeq  q \prod_{p \in \mathscr{P}(H)} p, \qquad
Q = Q(H;q,a) \defeq  q \prods{p \in \mathscr{P}(H)}{p \ne p_0} p, 
\end{align}
where
\begin{align}\label{(3.5)}
 \text{$p_0 = 1$ or $p_0$ is a prime satisfying $p_0 > \log H$.}
\end{align}
(The minor technical complication of $p_0$ has to be accounted for in the proof of \cite[Theorem 1]{GPY2006}, and consequently in the proof of \cite[Theorem 1.1]{F2010}. It arises when taking into consideration the possible existence of Siegel zeros --- see \cite{GPY2006} for details.) Finally, set
\begin{align}\label{(3.6)}
 \begin{split}
S = S(H;q,a) & \defeq \{1 \le h \le H : \text{$(Q,h) = 1$ and $h \equiv a \bmod q$}\}; \\
T = T(H;q,a) & \defeq \{1 \le h \le H : \text{$(Q,h) = 1$ and $h \not\equiv a \bmod q$}\}.
\end{split}
\end{align}

\begin{Lem}\label{Lemma 3.4}
Given a sufficiently large number $H$, and a coprime pair of integers $q \ge 3$ and $a$, let $Q = Q(H;q,a)$, $S = S(H;q,a)$, and $T = T(H;q,a)$ be as defined in \eqref{(3.1)} --- \eqref{(3.6)}. (a) For all sufficiently large $H$ we have, for integers $q$ satisfying 
\begin{align*}
3 \le q \le \frac{\log\log H}{\log\log\log H}
\end{align*}
and $a \equiv 1 \bmod q$, the inequality
\begin{align}\label{(3.7)}
 \ab{S} - \ab{T} \ge \frac{H}{\Gamma(1/\phi(q))} \br{\frac{\phi(Q)}{Q}}.
\end{align}
(b) For all sufficiently large $H$ we have, for integers $q$ satisfying 
\begin{align*}
3 \le q \le \frac{\log\log H}{2\log\log\log H}
\end{align*}
and $a \not\equiv 1 \bmod q$ coprime with $q$, the inequality
\begin{align}\label{(3.8)}
 \ab{S} - \ab{T} \ge \frac{2}{5} \cdot \frac{H}{(1 + \phi(q))\Gamma(1/\phi(q))} \br{\frac{\phi(Q)}{Q}}.
\end{align}
\end{Lem}

\begin{proof}[Proof of Theorem 1.1]
Fix an integer $q \ge 3$, arbitrary but bounded, and an integer $a$ that is coprime with $q$. Let $Q = Q(H;q,a)$, $S = S(H;q,a)$, and $T = T(H;q,a)$ be as defined in \eqref{(3.1)} --- \eqref{(3.6)}. In \cite[\S 6.2]{F2010}, we showed that there is a positive constant $A$, depending only on $q$, such that $\ab{S} - \ab{T} \gg_q H\phi(Q)/Q$ for some $H \in [Y/(\log Y)^{A},Y]$ and all sufficiently large $Y$. Using this, \emph{inter alia}, we established the lower bound \eqref{(1.1)} in \cite[\S 7]{F2010}. We also showed that if $\ab{S} - \ab{T} \gg_q H\phi(Q)/Q$ for all sufficiently large $H$, then \eqref{(1.2)} holds. Thus, Theorem \ref{Theorem 1.1} follows from Lemma \ref{Lemma 3.4}, in the way described in \cite[\S 7]{F2010}. (Here, the constant implied by $\gg_q$ depends only on $q$.)
\end{proof}

\begin{proof}[Proof of Lemma 3.4]
 Let $H$ be a sufficiently large number, and let a coprime pair of integers $q \ge 3$ and $a$ be given. Let $\mathscr{P}(H)$, $t(H)$, $\tilde{Q} = \tilde{Q}(H;q,a)$, $Q = Q(H;q,a)$, $p_0$, $S = S(H;q,a)$, and $T = T(H;q,a)$ be as defined in \eqref{(3.1)} --- \eqref{(3.6)}. We have
\begin{align}\label{(3.9)}
 \ab{T} \ll \frac{H}{\log H}.
\end{align}
This was shown in \cite[\S 6.2]{F2010}, where $q \ge 3$ was arbitrary but bounded. However, the larger $q$ is, the more primes there are that divide $Q$, hence the smaller the size of $T$. In \cite[\S 6.2]{F2010}, we actually bounded the size of $T$ by counting: the primes up to $H$; the integers of the form $pp'$, where $p \in (H/(\log H)^2, H]$ and $p' \in (\log H, (\log H)^2]$; the integers up to $H$ composed only of primes $p \le t(H)$ (using a result of de Bruijn on smooth numbers); and, in the case $p_0 \ne 1$, so that $p_0 > \log H$ by \eqref{(3.5)}, the multiples of $p_0$ up to $H$. Thus, \eqref{(3.9)} indeed holds uniformly for $q \ge 3$. 

Note that by definition of $\tilde{Q}$ and $Q$ (\eqref{(3.4)}, \eqref{(3.5)}), 
\begin{align}\label{(3.10)}
\ab{S} \ge \sumss{1 \le h \le H}{h \equiv 1 \bmod q}{(\tilde{Q},h) = 1} 1,
\end{align}
and
\begin{align}\label{(3.11)}
\phi(\tilde{Q})/\tilde{Q} \ge \textstyle\br{1 - \frac{1}{\log H}}\phi(Q)/Q.
\end{align}
We will work mainly with $\tilde{Q}$.

Now we suppose $q$ satisfies $3 \le q \le (\log H)^{\alpha}$, $\alpha \in (0,\frac{1}{2})$ given, and that $a \equiv 1 \bmod q$. Note that
\begin{align*}
 \log \br{\frac{H}{(\log H)^2}} = (\log H)\br{1 + \BigO{\frac{\log\log H}{\log H}}}.
\end{align*}
Thus, by definition of $\tilde{Q}$ (\eqref{(3.1)}, \eqref{(3.4)}), and two applications of Lemma \ref{Lemma 3.2}, 
\begin{align}\label{(3.12)}
\begin{split}
\phi(\tilde{Q})/\tilde{Q} & = 
\prod_{p \le H/(\log H)^2} \br{1 - \frac{1}{p}}
\prods{p \le H/(\log H)^2}{p \equiv 1 \bmod q}\br{1 - \frac{1}{p}}^{-1}
\prods{p \le \log H}{p \equiv 1 \bmod q}\br{1 - \frac{1}{p}} \\ & = 
\br{1 + \BigO{\frac{\log\log H}{\log H}}}
e^{-\gamma(1 - 1/\phi(q))}c(q) \frac{(\log H)^{\frac{1}{\phi(q)}}}{\log H}
\prods{p \le \log H}{p \equiv 1 \bmod q}\br{1 - \frac{1}{p}},
\end{split}
\end{align}
and so, by Lemma \ref{Lemma 3.3},
\begin{align}\label{(3.13)}
\begin{split}
\sumss{1 \le h \le H}{h \equiv 1 \bmod q}{(\tilde{Q},h) = 1} 1 & \ge
\sumss{1 \le h \le H}{p \mid h \Rightarrow p \equiv 1 \bmod q}{\text{and $p > \log H$}}\hspace{-5pt} 1 \\ & =
\br{1 + \BigO{  \frac{(\log\log H)^{c}}{ {\phantom{^{1-}}}  (\log H)^{1 - 2\alpha}}}  } 
\frac{c(q)}{\Gamma(1/\phi(q))}\cdot
\frac{H(\log H)^{\frac{1}{\phi(q)}}}{\log H} 
\hspace{-5pt} \prods{p \le \log H}{p \equiv 1 \bmod q}\hspace{-5pt}\br{1 - \frac{1}{p}} \\ & = 
\br{1 + \BigO{  \frac{(\log\log H)^{c}}{ {\phantom{^{1-}}}  (\log H)^{1 - 2\alpha}}}  } 
\frac{e^{\gamma(1 - 1/\phi(q))}}{\Gamma(1/\phi(q))}H \phi(\tilde{Q})/\tilde{Q}.
\end{split}
\end{align}
The left-hand side here is a lower bound for $\ab{S}$ \eqref{(3.10)}, so using the second line of \eqref{(3.13)} with the bound $\ab{T} \ll H/\log H$ \eqref{(3.9)}, we obtain
\begin{align*}
\frac{\ab{T}}{\ab{S}} \ll \frac{\Gamma(1/\phi(q))}{c(q)(\log H)^{\frac{1}{\phi(q)}}} 
\prods{p \le \log H}{p \equiv 1 \bmod q}\br{1 - \frac{1}{p}}^{-1}.
\end{align*}

At this point we suppose that $q$ is in the rather smaller range 
\begin{align*}
3 \le q \le \frac{\log\log H}{\log\log\log H}.
\end{align*}
Then we may apply Lemma \ref{Lemma 3.2} to this last product to obtain
\begin{align*}
\frac{\ab{T}}{\ab{S}} \ll \Gamma(1/\phi(q)) \br{\frac{\log\log H}{\log H}}^{\frac{1}{\phi(q)}} \ll 
\phi(q)\br{\frac{\log\log H}{\log H}}^{\frac{1}{\phi(q)}},
\end{align*}
and so
\begin{align*}
\log (\ab{T}/\ab{S}) & \le \BigO{1} + \log \phi(q) + \frac{1}{\phi(q)}\br{\log\log\log H - \log\log H} \\ & \le 
\BigO{1} + \log q + \frac{1}{q}\br{\log\log\log H - \log\log H} \\ & \le
\BigO{1} - \log \log \log \log H.
\end{align*}
Hence $\ab{T}/\ab{S} \ll 1/\log\log\log H$, and combining this with \eqref{(3.10)}, \eqref{(3.11)}, and \eqref{(3.13)}, we obtain
\begin{align*}
\ab{S} - \ab{T} \ge 
\br{1 + \BigO{\frac{1}{\log\log\log H}}} \frac{e^{\gamma(1 - 1/\phi(q))}}{\Gamma(1/\phi(q))} H \br{\frac{\phi(Q)}{Q}}.
\end{align*}
Noting that $e^{\gamma(1 - 1/\phi(q))} \ge e^{\gamma/2} > 1$ we obtain \eqref{(3.7)} (for all sufficiently large $H$), and the proof of part (a) is complete.

Now we suppose $a \not\equiv 1 \bmod q$. Once again we suppose that $3 \le q \le (\log H)^{\alpha}$, $\alpha \in (0,\frac{1}{2})$ given, until we want to show that $\ab{T}/\ab{S} = \littleo{1}$. Let us first of all show that
\begin{align}\label{(3.14)}
\begin{split}
\phi(\tilde{Q})/\tilde{Q} & = 
\br{1 + \BigO{\frac{\log\log\log H}{\log\log H}} } \\ & \hspace{70pt} \times
e^{-\gamma(1 - 1/\phi(q))} c(q)\cdot \frac{(\log t(H))^{\frac{1}{\phi(q)}}}{\log H} 
\prods{p \le \log H}{p \equiv 1 \bmod q} \br{1 - \frac{1}{p}}.
\end{split}
\end{align}
For by definition of $\tilde{Q}$ (\eqref{(3.2)} --- \eqref{(3.4)}),
\begin{align}\label{(3.15)}
\phi(\tilde{Q})/\tilde{Q} = 
\prod_{p \le H/(\log H)^2} \br{1 - \frac{1}{p}}
\prods{\log H < p \le t(H)}{p \equiv 1 \bmod q} \br{1 - \frac{1}{p}}^{-1}
\prods{H/t(H) < p \le H/(\log H)^2}{p \equiv a \bmod q} \br{1 - \frac{1}{p}}^{-1}.
\end{align}
By Mertens' theorem (the case $q = 1$ in Lemma \ref{Lemma 3.2}), 
\begin{align}\label{(3.16)}
\prod_{p \le H/(\log H)^2} \br{1 - \frac{1}{p}} & = 
\br{1 + \BigO{\frac{\log\log H}{\log H}} } \frac{ {\phantom{^-}} e^{-\gamma}}{\log H}.
\end{align}
Since $\log t(H) = (\log H)(\log\log\log H)/(2\log\log H)$ by definition \eqref{(3.2)} of $t(H)$, we certainly have $3 \le q \le (\log H)^{\frac{1}{2}} \le \log t(H)$ for all sufficiently large $H$, so applying Lemma \ref{Lemma 3.2} with $A = 1$, we obtain   
\begin{align}\label{(3.17)}
\begin{split}
\prods{\log H < p \le t(H)}{p \equiv 1 \bmod q} \br{1 - \frac{1}{p}}^{-1} & = 
\br{1 + \BigO{\frac{\log\log H}{(\log H)(\log\log\log H)} } } \\ & \hspace{50pt} \times
e^{\gamma/\phi(q)} c(q) (\log t(H))^{\frac{1}{\phi(q)}}   
\prods{p \le \log H}{p \equiv 1 \bmod q} \br{1 - \frac{1}{p}}.
\end{split}
\end{align}
As for the third product on the right-hand side of \eqref{(3.15)}, we have
\begin{align*}
1 & \le \prods{H/t(H) < p \le H/(\log H)^2}{p \equiv a \bmod q} \br{1 - \frac{1}{p}}^{-1} \le
\prod_{H/t(H) < p \le H/(\log H)^2} \br{1 - \frac{1}{p}}^{-1},
\end{align*}
and so two further applications of Mertens' theorem, plus a short calculation using the fact that $\log t(H) = (\log H)(\log\log\log H)/(2\log\log H)$, reveal that
\begin{align}\label{(3.18)}
1 & \le \prods{H/t(H) < p \le H/(\log H)^2}{p \equiv a \bmod q} \br{1 - \frac{1}{p}}^{-1} \le
1 + \BigO{\frac{\log\log\log H}{\log\log H}}.
\end{align}
Combining \eqref{(3.18)}, \eqref{(3.17)}, \eqref{(3.16)}, and \eqref{(3.15)} gives \eqref{(3.14)}.

Next, we will show that
\begin{align}\label{(3.19)}
\begin{split}
\sumss{1 \le h \le H}{h \equiv a \bmod q}{(\tilde{Q},h) = 1} 1 & \ge
\br{1 + \BigO{\frac{\log\log\log H}{\log\log H}}} \\ & \hspace{50pt} \times
\frac{\frac{1}{2}\br{1 - \frac{1}{e}}}{1 + \phi(q)} \cdot \frac{c(q)}{\Gamma(1/\phi(q))} \cdot \frac{H(\log t(H))^{\frac{1}{\phi(q)}}}{\log H} 
\prods{p \le \log H}{p \equiv 1 \bmod q} \br{1 - \frac{1}{p}}.
\end{split}
\end{align}
To this end we note, from the definition of $\tilde{Q}$ (\eqref{(3.3)}, \eqref{(3.4)}), that if $h = pm$, where $p > H/t(H)$ is a prime congruent to $a$ mod $q$, and $m \le H/p < t(H)$ is composed only of primes that are greater than $\log H$ and congruent to $1$ mod $q$, then $h \equiv a \bmod q$ and $(\tilde{Q},h) = 1$. We partition $(H/t(H),H]$ into sub-intervals
\begin{align*}
I_{l} \defeq (e^{l-1}H/t(H), e^{l}H/t(H)], \quad 1 \le l \le \log t(H),
\end{align*}
and deduce that
\begin{align}\label{(3.20)}
\sumss{1 \le h \le H}{h \equiv a \bmod q}{(\tilde{Q},h) = 1} 1 \ge
\sum_{1 \le l \le \log t(H)} \hspace{5pt}
\sums{p \in I_l}{p \equiv a \bmod q} \hspace{5pt}
\sumss{m \le t(H)/e^{l}}{p \mid m \Rightarrow p \equiv 1 \bmod q}{\text{and $p > \log H$}} 1.
\end{align}

Now, for $0 \le l \le \log t(H)$, we have
\begin{align*}
\log \br{\frac{e^{l}H}{t(H)}} =(\log H)\br{1 + \BigO{\frac{\log t(H)}{\log H}}} = (\log H)\br{1 + \BigO{\frac{\log\log\log H}{\log\log H}}},
\end{align*}
because $\log t(H) = (\log H)(\log\log\log H)/(2\log\log H)$ by definition \eqref{(3.2)} of $t(H)$. In particular, since $q \le (\log H)^{\alpha}$, $\alpha < \frac{1}{2}$, we certainly have $q \le \log (e^{l}H/t(H))$ for all sufficiently large $H$. Therefore we may apply Theorem \ref{Theorem 3.1} (Siegel-Walfisz), with $A = 1$, to obtain, for $1 \le l \le \log t(H)$, 
\begin{align}\label{(3.21)}
\begin{split}
\sums{p \in I_l}{p \equiv a \bmod q} 1 & = 
\sums{p \le e^{l}H/t(H)}{p \equiv a \bmod q} 1 - \sums{p \le e^{l-1}H/t(H)}{p \equiv a \bmod q} 1 \\ & = 
\br{1 + \BigO{\frac{\log\log\log H}{\log\log H}}}
\frac{1}{\phi(q)}\cdot \frac{H}{t(H)\log H}\br{1 - \frac{1}{e}}e^{l}.
\end{split}
\end{align}
Also, since $\log t(H) = (\log H)(\log\log\log H)/(2\log\log H)$ by definition \eqref{(3.2)} of $t(H)$, we have, for $1 \le l \le \frac{1}{2}\log t(H)$, that $\log H = \br{\log (t(H)/e^{l})}^{1 + \littleo{1}}$, where $\littleo{1}$ is shorthand for $\BigO{\log\log\log H/\log\log H}$. Thus, for $1 \le l \le \frac{1}{2}\log t(H)$ and all sufficiently large $H$, we have
\begin{align*}
3 \le q \le (\log H)^{\alpha} \le  \br{\log \br{t(H)/e^{l}}}^{\beta}, 
\quad \beta \defeq \textstyle\frac{1}{2}(\alpha + \frac{1}{2}) \in (0,\frac{1}{2}).
\end{align*}
Therefore, for $1 \le l \le \frac{1}{2}\log t(H)$, we may apply Lemma \ref{Lemma 3.3}, with $\beta$ in place of $\alpha$, and $A = 2$, say, to obtain
\begin{align}\label{(3.22)}
\begin{split}
\sumss{m \le t(H)/e^{l}}{p \mid m \Rightarrow p \equiv 1 \bmod q}{\text{and $p > \log H$}} 1 & = 
\br{1 + \BigO{\frac{\br{\log\log t(H)}^c  }{\br{\log t(H)}^{1 - 2\beta} }} } \frac{c(q)}{\Gamma(1/\phi(q))} \\ & 
\hspace{30pt} \times \frac{t(H)}{e^{l}} \cdot
\frac{\br{\log\br{t(H)/e^{l}}}^{\frac{1}{\phi(q)}}}{\log\br{t(H)/e^{l}}} 
\prods{p \le \log H}{p \equiv 1 \bmod q} \br{1 - \frac{1}{p}} \\ & \ge
\br{1 + \BigO{\frac{\br{\log\log t(H)}^c  }{\br{\log t(H)}^{1 - 2\beta} }} } 
\frac{c(q)}{\Gamma(1/\phi(q))} \\ & \hspace{30pt} \times 
\frac{t(H)}{e^{l}} \cdot
\frac{\br{\log t(H)}^{\frac{1}{\phi(q)}}}{\log t(H)}
\br{1 - \frac{l}{\log t(H)}}^{\frac{1}{\phi(q)}}
\prods{p \le \log H}{p \equiv 1 \bmod q} \br{1 - \frac{1}{p}}.
\end{split}
\end{align} 

Note that, since $\log t(H) = (\log H)(\log\log\log H)/(2\log\log H)$ by definition \eqref{(3.2)}, 
\begin{align*}
\frac{\br{\log\log t(H)}^c  }{\br{\log t(H)}^{1 - 2\beta} } \ll \frac{\log\log\log H}{\log\log H}
\end{align*}
for all sufficiently large $H$. Thus, combining \eqref{(3.22)} and \eqref{(3.21)} with \eqref{(3.20)}, we obtain
\begin{align}\label{(3.23)}
\begin{split}
\sumss{1 \le h \le H}{h \equiv a \bmod q}{(\tilde{Q},h) = 1} 1 & \ge
\sum_{1 \le l \le \frac{1}{2}\log t(H)} \hspace{5pt}
\sums{p \in I_l}{p \equiv a \bmod q} \hspace{5pt}
\sumss{m \le t(H)/e^{l}}{p \mid m \Rightarrow p \equiv 1 \bmod q}{\text{and $p > \log H$}} 1 \\ & \ge
\br{1 + \BigO{\frac{\log\log\log H}{\log\log H}}}
\frac{c(q)}{\Gamma(1/\phi(q))} \cdot
\frac{H(\log t(H))^{\frac{1}{\phi(q)}}}{\log H} 
\prods{p \le \log H}{p \equiv 1 \bmod q} \br{1 - \frac{1}{p}} \\ & \hspace{30pt} \times
\br{1 - \frac{1}{e}} \frac{1}{\phi(q)} \cdot \frac{1}{\log t(H)} 
\sum_{1 \le l \le \frac{1}{2}\log t(H)} \br{1 - \frac{l}{\log t(H)}}^{\frac{1}{\phi(q)}}.
\end{split}
\end{align}
Finally,
\begin{align*}
\sum_{1 \le l \le \frac{1}{2}\log t(H)} \br{1 - \frac{l}{\log t(H)}}^{\frac{1}{\phi(q)}} & \ge
\int_{1}^{\frac{1}{2}\log t(H)} \br{1 - \frac{u}{\log t(H)}}^{\frac{1}{\phi(q)}} \, \dd{u} \\ & = 
\frac{\log t(H)}{1 + \frac{1}{\phi(q)}}
\br{ \br{1 - \frac{1}{\log t(H)}}^{1 + \frac{1}{\phi(q)}} - \br{\frac{1}{2}}^{1 + \frac{1}{\phi(q)}}   } \\ & \ge
\frac{\log t(H)}{1 + \frac{1}{\phi(q)}}
\br{ \br{1 - \frac{1}{\log t(H)}}^{2} - \frac{1}{2}}    \\ & =
\frac{\log t(H)}{1 + \frac{1}{\phi(q)}} \cdot \frac{1}{2} \br{1 + \BigO{\frac{1}{\log t(H)}}},
\end{align*}
and combining this with \eqref{(3.23)} gives \eqref{(3.19)}.

Comparing \eqref{(3.14)} with \eqref{(3.19)}, then using \eqref{(3.10)} and \eqref{(3.11)}, we see that
\begin{align}\label{(3.24)}
\begin{split}
\ab{S} & \ge \sumss{1 \le h \le H}{h \equiv a \bmod q}{(\tilde{Q},h) = 1} 1 \\ & \ge
\br{1 + \BigO{\frac{\log\log\log H}{\log\log H}}}  
\frac{\frac{1}{2}\br{1 - \frac{1}{e}}}{1 + \phi(q)} \cdot
\frac{e^{\gamma(1 - 1/\phi(q))}}{\Gamma(1/\phi(q))} H\phi(\tilde{Q})/\tilde{Q} \\ & \ge
\br{1 + \BigO{\frac{\log\log\log H}{\log\log H}}}  
\frac{\frac{1}{2}\br{1 - \frac{1}{e}}}{1 + \phi(q)} \cdot
\frac{e^{\gamma(1 - 1/\phi(q))}}{\Gamma(1/\phi(q))} H\phi(Q)/Q.
\end{split}
\end{align}
Also, using the bound $\ab{T} \ll H/\log H$ \eqref{(3.9)}, and combining \eqref{(3.10)} with \eqref{(3.19)}, we obtain
\begin{align*}
\frac{\ab{T}}{\ab{S}} \ll \frac{\phi(q)\Gamma(1/\phi(q))}{c(q)(\log t(H))^{\frac{1}{\phi(q)}}} 
\prods{p \le \log H}{p \equiv 1 \bmod q}\br{1 - \frac{1}{p}}^{-1}.
\end{align*}

At this point we suppose that $q$ is in the rather smaller range 
\begin{align*}
3 \le q \le \frac{1}{2}\cdot \frac{\log\log H}{\log\log\log H}.
\end{align*}
Then we may apply Lemma \ref{Lemma 3.2} to this last product to obtain
\begin{align*}
\frac{\ab{T}}{\ab{S}} \ll \phi(q)\Gamma(1/\phi(q)) \br{\frac{\log\log H}{\log t(H)}}^{\frac{1}{\phi(q)}} \ll 
(\phi(q))^2\br{\frac{\log\log H}{\log t(H)}}^{\frac{1}{\phi(q)}},
\end{align*}
and so, since $\log t(H) = (\log H)(\log \log \log H)/(2\log\log H)$ by \eqref{(3.2)}, 
\begin{align*}
\log (\ab{T}/\ab{S}) & \le \BigO{1} + 2\log \phi(q) + \frac{1}{\phi(q)}\br{\log\log\log H - \log\log t(H)} \\ & \le
\BigO{1} + 2\log q \\ & \hspace{30pt} + \frac{1}{q}\br{2\log\log\log H -\log\log H -\log\log\log\log H  + \BigO{1}} \\ & \le
\BigO{1} - 2\log \log \log \log H.
\end{align*}
Hence $\ab{T}/\ab{S} \ll 1/(\log\log\log H)^2$, and, combining this with the last inequality in \eqref{(3.24)}, we obtain
\begin{align*}
\ab{S} - \ab{T} \ge \br{1 + \BigO{\frac{1}{(\log\log\log H)^2}}}
\frac{\frac{1}{2}\br{1 - \frac{1}{e}}}{1 + \phi(q)} \cdot
\frac{e^{\gamma(1 - 1/\phi(q))}}{\Gamma(1/\phi(q))} H\phi(Q)/Q.
\end{align*}
Noting that $\frac{1}{2}\br{1 - \frac{1}{e}}e^{\gamma(1 - 1/\phi(q))} \ge \frac{1}{2}\br{1 - \frac{1}{e}}e^{\gamma/2} = 0.42\ldots > \frac{2}{5}$, we obtain \eqref{(3.8)} (for all sufficiently large $H$), and the proof of part (b) is complete.
\end{proof}

\section{Proof of Lemmas 3.2 and 3.3} \label{Section 3}

Throughout this section, at each and every occurrence of $\Oh$, $\ll$, and $\gg$, the implied constant is absolute. When we write $A \asymp B$, we mean $A \ll B$ and $B \ll A$ both hold, the implied constants being absolute. The letter $c$, by itself, always denotes an absolute positive constant, possibly a different constant at each occurrence.  Also, $s = \sigma + i\tau$ denotes a complex variable, $\sigma$ and $\tau$ being real. We often use $s$ and $\sigma + i\tau$ interchangeably: for example if we write $\ab{X^s} = X^{\sigma}$, or $\zeta(s) \ne 0$ for $\ab{\tau} \ge 2$, $\sigma \ge 1 - 1/\log\ab{\tau}$, it is to be understood that $s = \sigma + i\tau$. We always assume $q$ is an integer and that $q \ge 3$. A Dirichlet character $\chi$ is always to be taken as a character to the modulus $q$, with corresponding Dirichlet $L$-function $L(s,\chi)$. Finally, $\chi_0$ always denotes the principal character to the modulus $q$.

The following proof of Lemma \ref{Lemma 3.2} is, \emph{mutatis mutandis}, a proof due to Hardy \cite{H1927} of Mertens' theorem (the case $q = 1$), in which the following propositions are used:
\begin{Prop}\label{Proposition 4.1}
If $a > 0$, $\delta > 0$, and $\delta \to 0$, then
\begin{align*}
\int_a^{\infty} \frac{e^{-\delta t}}{t} \, \dd{t} - \log \frac{1}{\delta} \to -\log a - \gamma
\end{align*}
when $\delta \to 0$, where $\gamma = 0.57721\ldots$ is the Euler-Mascheroni constant.
\end{Prop}
\begin{proof}
This is Proposition (A) in \cite{H1927}, where it is not proved, but is said to be ``familiar in the theory of the Gamma function''.
\end{proof}

\begin{Prop}\label{Proposition 4.2}
If 
\begin{align*}
 J(\delta) = \int_a^{\infty} f(t) t^{-\delta} \, \dd{t}
\end{align*}
where $a > 0$, is convergent for $\delta > 0$ and tends to a limit $l$ when $\delta \to 0$, and 
\begin{align*}
f(t) = \BigO{\frac{1}{t\log t}},
\end{align*}
then $J(0)$ is convergent and has the value $l$. 
\end{Prop}
\begin{proof}
This result of Landau \cite{L1913} is Proposition (D) in \cite{H1927}.
\end{proof}

In addition to Propositions \ref{Proposition 4.1} and \ref{Proposition 4.2}, we will use some basic properties of Dirichlet $L$-series. In particular, for $\chi \ne \chi_0$, $L(s,\chi)$ is analytic and converges for $\sigma > 0$; it is absolutely convergent for $\sigma > 1$. We have $L(1,\chi) > 0$ for $\chi \ne \chi_0$ (see Theorem \ref{Theorem 4.4} below). We have 
\begin{align}\label{(4.1)}
L(s,\chi_0) = \zeta(s) \prod_{p \mid q}\br{1 - \frac{1}{p^{s}}}, \quad \sigma > 1,
\end{align}
where $\zeta(s)$ is the Riemann zeta function, analytic throughout the complex plane except for a simple pole at $s = 1$, where the residue is $1$; in fact \cite[(2.1.16)]{T1986}
\begin{align}\label{(4.2)}
 \zeta{(s)} = \frac{1}{s-1} + \gamma + \BigO{\ab{s-1}}, \quad s \to 1.
\end{align}

Because of the orthogonality relation \cite[II \S 8.1, Theorem 2 (a)]{T1995}:
\begin{align}\label{(4.3)}
 \sum_{\chi \bmod q} \chi(n) = 
\begin{cases}
 \phi(q) & \text{if $n \equiv 1 \bmod q$,} \\
 0 & \text{otherwise,}
\end{cases}
\end{align}
we have, for $\sigma > 1$,
\begin{align}\label{(4.4)}
\begin{split}
\log \textstyle \prod_{\chi \bmod q} L(s,\chi) & = 
\sum_{\chi \bmod q} \log \textstyle \prod_p \br{1 - \frac{\chi(p)}{p^{s}}}^{-1} \\ & = 
\sum_{\chi \bmod q} \sum_{p} \sum_{m=1}^{\infty} \frac{\chi(p^m)}{mp^{ms}} \\ & = 
\sum_{p} \sum_{m=1}^{\infty} \frac{1}{mp^{ms}} \sum_{\chi \bmod q} \chi(p^m) \\ & = 
\phi(q)\sum_{p} \sums{m=1}{p^m \equiv 1 \bmod q}^{\infty} \frac{1}{mp^{ms}}. 
\end{split}
\end{align}
Exponentiating, we deduce that for $\sigma > 1$, $\prod_{\chi \bmod q} L(\sigma,\chi)$ is a real number greater than or equal to $1$. Also, we see that for $\sigma > 1$,
\begin{align}\label{(4.5)}
\begin{split}
\sum_{p \equiv 1 \bmod q} \log \br{1 - \frac{1}{p^{s}}}^{-1} & = 
\sum_{p} \sum_{m=1}^{\infty} \frac{1}{mp^{ms}} \\ & = 
\log \Theta(s) + \frac{1}{\phi(q)}\log \textstyle \prod_{\chi \bmod q} L(s,\chi),
\end{split} 
\end{align}
where
\begin{align}\label{(4.6)}
\Theta(s) \defeq \exp\Bigg(-\sum_{p \not\equiv 1 \bmod q} \hspace{5pt} \sums{m=2}{p^m \equiv 1 \bmod q}^{\infty} \frac{1}{mp^{ms}}\Bigg)
\end{align}
is absolutely convergent and analytic for $\sigma > 1/2$, and $\Theta(1) \asymp 1$ (see the proof of Lemma \ref{Lemma 4.5} (a), below). 

We will also use Theorem \ref{Theorem 3.1} (Siegel-Walfisz), and therefore, implicitly, certain properties and results concerning $L$-functions that are used in its proof, in the proof of Lemma \ref{Lemma 3.2}.

\begin{proof}[Proof of Lemma 3.2]
Let $X \ge 3$ be a number and let $q \ge 3$ be an integer. We use the notation
\begin{align*}
\pi(t;q,1) \defeq \sums{p \le t}{p \equiv 1 \bmod q} 1.
\end{align*}
We begin by noting that for $\sigma \ge 1$, partial summation gives
\begin{align}\label{(4.7)}
 \sums{p \le X}{p \equiv 1 \bmod q} \log\br{1 - \frac{1}{p^{\sigma}}}^{-1} = 
\pi(X;q,1)\log\br{1 - \frac{1}{X^{\sigma}}}^{-1} + \sigma\int_{e}^{X} \frac{\pi(t;q,1)}{t(t^{\sigma} - 1)} \, \dd{t}.
\end{align}
Now $\pi(X;q,1)\log\br{1 - X^{-\sigma}}^{-1} \ll X^{1 - \sigma}$, and so, for $\sigma > 1$, letting $X \to \infty$ in \eqref{(4.7)} gives
\begin{align*}
 \sum_{p \equiv 1 \bmod q} \log\br{1 - \frac{1}{p^{\sigma}}}^{-1} & =
\sigma \int_{e}^{\infty} \frac{\pi(t;q,1)}{t(t^{\sigma} - 1)}\, \dd{t} \\ & = 
\sigma \int_{e}^{\infty} \br{\pi(t;q,1) - \frac{t}{\phi(q)\log t}}t^{-1 - \sigma}  \, \dd{t} \\ & \hspace{50pt} + 
\sigma \int_{e}^{\infty} \frac{\pi(t;q,1)}{t^{1 + \sigma}(t^{\sigma} - 1)} \, \dd{t} + 
 \int_{e}^{\infty} \frac{ {\phantom{^{-\sigma}}}  t^{-\sigma}}{\phi(q)\log t} \, \dd{t} \\ & \eqdef 
J_1(\sigma) + J_2(\sigma) + J_3(\sigma).
\end{align*}

On the other hand, with $\Theta(\sigma)$ as defined in \eqref{(4.6)} we have, by \eqref{(4.5)} and \eqref{(4.1)},
\begin{align*}
\sum_{p \equiv 1 \bmod q} \log\br{1 - \frac{1}{p^{\sigma}}}^{-1} & = \log \Theta(\sigma)  +
\frac{1}{\phi(q)}\log \zeta(\sigma) \\ & \hspace{30pt} +
 \frac{1}{\phi(q)}\textstyle\log \br{\prod_{p \mid q} \br{1 - \frac{1}{p^{\sigma}}}\prod_{\chi \ne \chi_0} L(\sigma,\chi)}.
\end{align*}
Comparing the last two expressions for $\sum_{p \equiv 1 \bmod q}\log(1 - p^{-\sigma})^{-1}$, we obtain, for $\sigma > 1$,
\begin{align}\label{(4.8)}
\begin{split}
  J_1(\sigma) & = - J_2(\sigma) - J_3(\sigma) +
\frac{1}{\phi(q)} \log \zeta(\sigma) \\ & \hspace{30pt} + 
\log\Theta(\sigma) +
\frac{1}{\phi(q)}\textstyle\log \br{\prod_{p \mid q} \br{1 - \frac{1}{p^{\sigma}}}\prod_{\chi \ne \chi_0} L(\sigma,\chi)}.
\end{split}
\end{align}

As $\sigma \to 1^{+}$, $J_2(\sigma) \to J_2(1)$ by uniform convergence;  
\begin{align*}
 J_3(\sigma) = \frac{\sigma}{\phi(q)}\int_{1}^{\infty} e^{-(\sigma - 1)u}\frac{\dd{u}}{u} = 
\frac{1}{\phi(q)}\br{\log \br{\frac{1}{\sigma - 1}} - \gamma + \littleo{1}},
\end{align*}
by Proposition \ref{Proposition 4.1}; and

\begin{align*}
 \frac{1}{\phi(q)}\log\zeta{(\sigma)} = \frac{1}{\phi(q)}\br{\log \br{\frac{1 + \littleo{1}}{\sigma - 1}}} = 
 \frac{1}{\phi(q)}\br{\log \br{\frac{1}{\sigma - 1}} + \littleo{1}},
\end{align*} 
by \eqref{(4.2)}. Combining all of this with \eqref{(4.8)}, we see that, as $\sigma \to 1^{+}$, 
\begin{align*}
 J_1(\sigma) \to l \defeq -J_2(1) + \frac{\gamma}{\phi(q)} + 
\log\Theta(1) +
\frac{1}{\phi(q)} \log \br{ (\phi(q)/q) \textstyle\prod_{\chi \ne \chi_0} L(1,\chi)}.
\end{align*}
Applying Proposition \ref{Proposition 4.2}, with $\delta = \sigma - 1$, and
\begin{align*}
f(t) = \br{\pi(t;q,1) - \frac{t}{\phi(q)\log t}}t^{-2} \ll \br{\sum_{p \le t} 1 + \frac{t}{\log t}}t^{-2} \ll \frac{1}{t\log t}
\end{align*}
(by the prime number theorem), we see that $J_1(1)$ is convergent and has value $l$. Thus
\begin{align}\label{(4.9)}
 J_1(1) + J_2(1) = \frac{\gamma}{\phi(q)} + 
\log\Theta(1) +
\frac{1}{\phi(q)} \log \br{ (\phi(q)/q) \textstyle\prod_{\chi \ne \chi_0} L(1,\chi)}.
\end{align}

Now, supposing $3 \le q \le (\log X)^A$ for some positive number $A$, Theorem \ref{Theorem 3.1} (Siegel-Walfisz) implies that
\begin{align*}
 \br{\pi(t;q,1) - \frac{t}{\phi(q)\log t}}t^{-2} \ll \frac{1}{t(\log t)^2}
\end{align*}
for all $t \ge X$, and so
\begin{align}\label{(4.10)}
\begin{split}
 \int_{e}^{X} \br{\pi(t;q,1) - \frac{t}{\phi(q)\log t}}t^{-2}  \, \dd{t} & = 
J_1(1) - \int_{X}^{\infty} \br{\pi(t;q,1) - \frac{t}{\phi(q)\log t}}t^{-2}  \, \dd{t} \\ & = 
J_1(1) + \BigO{\frac{1}{\log X}}.
\end{split}
\end{align}
Also, since, by the prime number theorem, $\pi(t;q,1)/(t^2(t-1)) \ll 1/(t^2\log t)$ for all $q \ge 3$ and $t \ge e$, we have
\begin{align}\label{(4.11)}
 \int_{e}^{X} \frac{\pi(t;q,1)}{t^{2}(t - 1)} \, \dd{t} = 
J_2(1) - \int_{X}^{\infty} \frac{\pi(t;q,1)}{t^{2}(t - 1)} \, \dd{t} = 
J_2(1) + \BigO{\frac{1}{X\log X}}.
\end{align}

We return at last to \eqref{(4.7)}, in which we now take $\sigma = 1$, $X$ a sufficiently large number, and $q$ an integer in the range $3 \le q \le (\log X)^A$. Using the fact that $\pi(X;q,1) \ll X/\log X$ by the prime number theorem, followed by \eqref{(4.10)}, and  \eqref{(4.11)}, we obtain

\begin{align*}
 \sums{p \le X}{p \equiv 1 \bmod q} \log\br{1 - \frac{1}{p}}^{-1} & =
\pi(X;q,1)\log\br{1 - \frac{1}{X}}^{-1} + \int_{e}^{X} \frac{\pi(t;q,1)}{t(t-1)} \, \dd{t} \\ & = 
\BigO{\frac{1}{\log X}} + 
\int_{e}^{X} \br{\pi(t;q,1) - \frac{t}{\phi(q)\log t}}t^{-2}  \, \dd{t} \\ & \hspace{50pt} + 
\int_{e}^{X} \frac{\pi(t;q,1)}{t^{2}(t - 1)} \, \dd{t} + 
\frac{1}{\phi(q)}\int_{e}^{X} \frac{ \dd{t}}{t\log t} \\ & = 
\BigO{\frac{1}{\log X}} + J_1(1) + J_2(1) + \frac{1}{\phi(q)}\log\log X.
\end{align*}
Using \eqref{(4.9)} and exponentiating, we obtain 
\begin{align*}
\prods{p \le X}{p \equiv 1 \bmod q} \br{1 - \frac{1}{p}}^{-1} = 
\br{1 + \BigO{\frac{1}{\log X}}} e^{\gamma/\phi(q)}c(q) (\log X)^{\frac{1}{\phi(q)}},
\end{align*}
with
\begin{align}\label{(4.12)}
c(q) \defeq \Theta(1) \br{\frac{\phi(q)}{q}\prod_{\chi \ne \chi_0} L(1,\chi)}^{1/\phi(q)}.
\end{align}
By the observations made before the proof, $\Theta(1)$ is positive, and the product $\prod_{\chi \ne \chi_0} L(1,\chi)$ is real and positive, so we may indeed take the real, positive $\phi(q)$th root here. Hence $c(q)$ is real and positive. 
\end{proof}

We now proceed with the proof of Lemma \ref{Lemma 3.3}. We will use the following estimate:

\begin{Thm}[Effective Perron formula]\label{Theorem 4.3}
 Let $F(s) \defeq \sum_{n = 1}^{\infty} a_n n^{-s}$ be a Dirichlet series with abscissa of absolute convergence $\sigma_a$. For $\kappa > \max(0,\sigma_a)$, $T \ge 1$, and $X \ge 1$, we have
\begin{align*}
 \sum_{n \le X} a_n = 
\frac{1}{2\pi i} \int_{\kappa - i\infty}^{\kappa + i\infty} F(s) \frac{{\phantom{^{s}}} X^s}{s} \, \dd{s}  +
\BigO{X^{\kappa} \sum_{n=1}^{\infty} \frac{\ab{a_n}}{n^{\kappa}(1 + T\ab{\log(X/n)})}}.
\end{align*}
If $X \ge 2$, $0 < \kappa - 1 \ll 1/\log X$, and $a_n \ll 1$ for every $n$, then
\begin{align}\label{(4.13)}
 \sum_{n \le X} a_n = 
\frac{1}{2\pi i} \int_{\kappa - iT}^{\kappa + iT} F(s) \frac{{\phantom{^{s}}} X^s}{s} \, \dd{s}  +
\BigO{1} + \BigO{(X\log X)/T}.
\end{align}
\end{Thm}
\begin{proof}
The first estimate is the effective Perron formula (see \cite[II \S 2.1, Theorem 2]{T1995}). From this, assuming $X \ge 2$, $0 < \kappa - 1 \ll 1/\log X$, and $a_n \ll 1$ for every $n$, we deduce \eqref{(4.13)} as follows. We partition the sum of the first $\Oh$-term into three, according as $n \le X/2$, $X/2 < n \le 2X$, or $n > 2X$. Noting that $X^{\kappa} \ll X$, we have
\begin{align*}
 X^{\kappa} \sum_{n \le X/2} \frac{\ab{a_n}}{n^{\kappa}(1 + T\ab{\log(X/n)})} \ll
\frac{X}{T} \sum_{n \le X} \frac{1}{n} \ll \frac{X\log X}{T},
\end{align*}
and
\begin{align*}
X^{\kappa} \sum_{n > 2X} \frac{\ab{a_n}}{n^{\kappa}(1 + T\ab{\log(X/n)})} \ll 
\frac{X}{T}\sum_{n > X} \frac{1}{n^{\kappa}} \ll 
\frac{X}{T} \int_{X}^{\infty} \frac{\dd t}{t^{\kappa}} \ll 
\frac{X}{T(\kappa - 1)X^{\kappa - 1}} \ll 
\frac{X\log X}{T}.
\end{align*}
For $X/2 < n \le 2X$, we use
\begin{align*}
 \ab{\log (X/n)} = \abs{\log \br{1 + \frac{X-n}{n}}} >
\begin{cases}
 \displaystyle\frac{1}{2}\br{\frac{X-n}{n}} & \text{if $n < X$,} \\
& \\
\displaystyle\frac{n - X}{n} & \text{if $n > X$.}
\end{cases}
\end{align*}
We have 
\begin{align*}
X^{\kappa} \sum_{X/2 < n \le 2X} \frac{1}{n^{\kappa}(1 + T\ab{\log(X/n)})} & \le
\frac{X^{\kappa}}{(X-1/2)^{\kappa}} + 
\frac{{\phantom{^{\kappa}}}X^{\kappa}}{T} \sums{X/2 < n \le 2X}{\ab{n - X} > 1/2} \frac{n}{n^{\kappa}\ab{n - X}} \\ & \ll
1 +  \frac{X}{T} \sums{X/2 < n \le 2X}{\ab{n - X} > 1/2} \frac{1}{\ab{n - X}} \\ & \ll 
1+ \frac{X}{T} \sum_{n \le 2X} \frac{1}{n} \\ & \ll
1+ (X\log X)/T.
\end{align*}
Combining, we obtain the error term in \eqref{(4.13)}.
\end{proof}

Let us now gather some more properties of Dirichlet characters and $L$-series that will be used in the proof of Lemma \ref{Lemma 3.3}.

\begin{Thm}\label{Theorem 4.4}
There is an effectively computable positive constant $\bar{c}$ such that the following holds for any given integer $q \ge 3$. The product $\prod_{\chi \bmod q} L(s,\chi)$ has at most one zero in the region 
\begin{align*}
 \mathcal{D} \defeq 
\Br{\sigma + i\tau : \sigma \ge 1 - \frac{\bar{c}}{\log \max (q, q\ab{\tau})}}.
\end{align*}
Such a zero, if it exists, is real and simple, and corresponds to a non-principal real character. 
\end{Thm}
\begin{proof}
See \cite[Chapter 14]{D2000}. As far as explicit constants $\bar{c}$ go, the best result is due to Kadiri \cite[Theorem 1.1]{K2007}, who has shown that $\bar{c} = 1/6.41$ is admissible.
\end{proof}

As noted earlier, the $L$-series $L(s,\chi)$, $\chi \ne \chi_0$, converge for $\sigma > 0$. Indeed, the P\'olya-Vinogradov inequality \cite[Chapter 23]{D2000}:
\begin{align*}
\max_{t \ge 1} \Big | \sum_{n \le t} \chi(n) \Big | \le q^{\frac{1}{2}}\log q, \quad \chi \ne \chi_0,
\end{align*}
together with partial summation, gives 
\begin{align}\label{(4.14)}
L(s,\chi) & \ll 
(q^{\frac{1}{2}}\log q)  \br{1 + \ab{s} \int_1^{\infty} \frac{\dd{t}}{t^{\sigma + 1}}}   \ll
(q^{\frac{1}{2}}\log q) \br{1 + \ab{s}/\sigma}, \quad \sigma > 0, \quad \chi \ne \chi_0.
\end{align}
Thus, the product $\prod_{\chi \bmod q} L(s,\chi)$ is analytic at every point in the region $\mathcal{D}$ of Theorem \ref{Theorem 4.4}, except for a simple pole at $s = 1$. This is due to the Riemann zeta function, which is analytic except for a simple pole at $s = 1$, where the residue is one. More precisely, we have \eqref{(4.2)}. Let us also recall here that
\begin{align}\label{(4.15)}
 \zeta(\sigma + i\tau) \ll \log\ab{\tau}, \quad \ab{\tau} \ge 2, \quad \sigma \ge 1 - c/\log\ab{\tau}.
\end{align}

By \eqref{(4.1)}, we have
\begin{align*}
\prod_{\chi \bmod q} L(s,\chi) = \zeta(s)\prod_{p \mid q}\br{1 - \frac{1}{p^s}} \prod_{\chi \ne \chi_0} L(s,\chi),
\end{align*}
and, by Theorem \ref{Theorem 4.4}, $L(1,\chi) \ne 0$ for $\chi \ne \chi_0$. In fact, we have  
\begin{align}\label{(4.16)}
\begin{split}
L(1,\chi) \gg 
\begin{cases}
q^{-\frac{1}{2}}        & \chi^2 = \chi_0, \\
(\log q)^{-7}           & \chi^2 \ne \chi_0.
\end{cases}
\end{split}
\end{align}
(See \cite[\S 11.3, Theorem 11.11]{MV2007} and \cite[\S II.8, Theorem 8]{T1995} respectively.) Furthermore, as we have already observed, as a consequence of \eqref{(4.4)}, we have $\prod_{\chi \bmod q} L(\sigma,\chi) \ge 1$ for $\sigma > 1$. 

Thus, if we define a simply connected domain
\begin{align}\label{(4.17)}
 \mathcal{D}^{*} \defeq 
\displaystyle\Br{\sigma + i\tau : \sigma \ge 1 - \frac{\bar{c}}{\log \max (q, q\ab{\tau})}} \Big\backslash
\left[1 - \frac{\bar{c}}{\log q}, 1 \right],
\end{align}
we see that the function
\begin{align*}
 \br{(s-1) \prod_{\chi \bmod q} L(s,\chi)}^{1/\phi(q)}
\end{align*}
is analytic throughout $\mathcal{D}^{*} \cup \{1\}$, whereas the function
\begin{align*}
 \br{\prod_{\chi \bmod q} L(s,\chi)}^{1/\phi(q)} = 
(s-1)^{-1/\phi(q)}\br{(s-1) \prod_{\chi \bmod q} L(s,\chi)}^{1/\phi(q)}
\end{align*}
has a branch point at $s = 1$, but is analytic throughout $\mathcal{D}^{*}$. We always choose the principal value of the complex logarithm, so that $\lim_{s \to 1}((s-1)\zeta(s))^{1/\phi(q)} = 1$, for instance.

These functions, slightly modified, feature in Lemma \ref{Lemma 4.5} below. Given an integer $q \ge 3$ and a number $Y \ge 1$, we define $\Theta(s)$, for $\sigma > 1/2$, as in \eqref{(4.6)}; then we define
\begin{align}\label{(4.18)}
 G(s) \defeq 
\Theta(s)
\br{(s-1)\prod_{\chi \bmod q} L(s,\chi)}^{1/\phi(q)}
\prods{p \le Y}{p \equiv 1 \bmod q}\br{1 - \frac{1}{p^s}},
\quad s \in \mathcal{D}^{*} \cup \{1\};
\end{align}
and
\begin{align}\label{(4.19)}
 F(s) \defeq (s-1)^{-1/\phi(q)}G(s),
\quad s \in \mathcal{D}^{*}.
\end{align}
Also, given $\sigma > 0$, we define
\begin{align}\label{(4.20)}
\Pi_1(\sigma; q) & \defeq \prod_{p \mid q} \br{1 + \frac{1}{p^{\sigma}}}, 
\quad
\Pi_2(\sigma; Y) \defeq \prods{p \le Y}{p \equiv 1 \bmod q} \br{1 + \frac{1}{p^{\sigma}}}.
\end{align}
Note that if $q \ll 1$ then $\Pi_1(\sigma; q) \ll 1$. Otherwise, since $\log(1 + x) \le x$ for $x \ge 0$, we have
\begin{align}\label{(4.21)}
\Pi_1(\sigma; q) \le
\exp\Bigg(  \sum_{p \mid q} \frac{1}{p^{\sigma}}  \Bigg)  \le
\exp\Bigg(  q^{1 - \sigma} \sum_{p \mid q} \frac{1}{p}  \Bigg)  \le
\exp\br{cq^{1 - \sigma}\log\log\log q}
\end{align}
by a standard estimate. Similarly, if $Y \ll 1$, or if $Y \le q$, then $\Pi_2(\sigma; Y) \ll 1$, and otherwise
\begin{align}\label{(4.22)}
\Pi_2(\sigma; Y) \le
\exp\Bigg(  \sum_{p \le Y} \frac{1}{p^{\sigma}}  \Bigg)  \le
\exp\Bigg(  Y^{1 - \sigma} \sum_{p \le Y} \frac{1}{p}  \Bigg)  \le
\exp\br{cY^{1 - \sigma}\log\log Y}.
\end{align}

\begin{Lem}\label{Lemma 4.5}
Fix an integer $q \ge 3$ and a number $Y \ge 1$, and let $\mathcal{D}^{*}$, $\Theta(s)$, $G(s)$, $F(s)$, $\Pi_1(\sigma;q)$, and $\Pi_2(\sigma; Y)$ be as defined in \eqref{(4.6)}, \eqref{(4.17)} --- \eqref{(4.20)}.
(a) The function $G(s)$ is analytic throughout $\mathcal{D}^{*} \cup \{1\}$. For $s \in \mathcal{D}^{*}$ with $\ab{\tau} \ge 2$,  we have
\begin{align}\label{(4.23)}
 G(s)   \ll 
(\log\ab{\tau})^{1/\phi(q)} 
(1 + \ab{\tau})^{1 - 1/\phi(q)} 
(q^{\frac{1}{2}}\log q) \Pi_1(\sigma;q) \Pi_2(\sigma;Y).  
\end{align}
For $s \in \mathcal{D}^{*} \cup \{1\}$ with $s - 1 \ll 1$, we have
\begin{align}\label{(4.24)}
G(s) \ll (q^{\frac{1}{2}}\log q) \Pi_1(\sigma;q) \Pi_2(\sigma;Y).
\end{align}
Also,
\begin{align}\label{(4.25)}
 G(s) = G(1) + 
\BigO{ \ab{s-1} (q^{\frac{1}{2}}\log q) \Pi_1(\sigma;q) \Pi_2(\sigma;Y) }.
\end{align}
(b) The function $F(s)$ is analytic at every point in $\mathcal{D}^{*}$. We have
\begin{align}\label{(4.26)}
 F(s) = \prod_{p \equiv 1 \bmod q} \br{1 - \frac{1}{p^s}}^{-1}
\prods{p \le Y}{p \equiv 1 \bmod q} \br{1 - \frac{1}{p^s}}, \quad \sigma > 1.
\end{align}
\end{Lem}
\begin{proof}
(a) For any $s$ with $\sigma \ge 1/2 + \delta > 1/2$, we have, by definition \eqref{(4.6)} of $\Theta(s)$, the bound
 \begin{align*}
\ab{\log \Theta(s)} & \le
\sum_{p} \sum_{m = 2}^{\infty} \frac{1}{p^{m \sigma}}  \le 
\sum_{p} \sum_{m = 2}^{\infty} \frac{1}{p^{m(1/2 + \delta)}}  = 
\sum_{p} \frac{1/p^{1 + 2\delta}}{1 - 1/p^{1/2 + \delta}} \\ & \le
\frac{1}{1 - 1/2^{1 + 2\delta}}\sum_{n=1}^{\infty} \frac{1}{n^{1 + 2\delta}}  \le
\frac{1}{1 - 1/2^{1/2 + \delta}} \br{ 1 + \int_{1}^{\infty} \frac{\dd t}{t^{1 + 2\delta}}    } \\  & = 
\frac{1}{1 - 1/2^{1/2 + \delta}} \br{1 + \frac{1}{2\delta}}.
 \end{align*}
Denoting this last expression by $c(\delta)$, we have $\ab{\Theta(s)} \le e^{\ab{\log \Theta(s)}} \le e^{c(\delta)}$, and likewise $1/\ab{\Theta(s)} \le e^{\ab{-\log \Theta(s)}} \le e^{c(\delta)}$. Thus, uniformly for $\sigma \ge 1/2 + \delta$, we have 
\begin{align}\label{(4.27)}
 e^{-c(\delta)} \le \ab{\Theta(s)} \le e^{c(\delta)}, 
\quad c(\delta) \defeq \frac{1}{1 - 1/2^{1/2 + \delta}} \br{1 + \frac{1}{2\delta}},
\end{align}
and so we can differentiate the series for $\log \Theta(s)$ term by term by the uniform convergence theorem. Indeed,
\begin{align*}
\frac{\dd{^k}}{\dd{s^k}} \log \Theta(s) = 
-\sum_{p \not\equiv 1 \bmod q} \sums{m=2}{p^m \equiv 1 \bmod q}^{\infty} \frac{(-\log p^m)^k}{mp^{ms}}, \quad \sigma > 1/2, \quad k \ge 1.
\end{align*}
We see that $\log\Theta(s)$, and therefore $\Theta(s) = \exp\br{\log \Theta(s)}$, is analytic throughout the half-plane $\sigma > 1/2$. 

Note in particular that \eqref{(4.27)} implies $\Theta(s) \ll 1$ on $\mathcal{D}^{*} \cup \{1\}$. By our earlier discussion, the term $(\cdots)^{1/\phi(q)}$ in the definition \eqref{(4.18)} of $G(s)$ is analytic on $\mathcal{D}^{*} \cup \{1\}$. The final product in \eqref{(4.18)}, being finite, is non-zero and uniformly bounded by $\prod_{p \le Y} (1 + p^{-1/2})$, say, for $s \in \mathcal{D}^{*} \cup\{1\}$. 

Hence $G(s)$ is analytic on $\mathcal{D}^{*} \cup \{1\}$, where by the P\'olya-Vinogradov inequality \eqref{(4.14)}, and since $\sigma \gg 1$ throughout $\mathcal{D}^{*}$, we have (recalling the definition \eqref{(4.20)} of $\Pi_2(\sigma, Y)$):
\begin{align*}
 G(s) & \ll \br{(s - 1) \textstyle\prod_{\chi \bmod q} L(s,\chi) }^{1/\phi(q)} \Pi_2(\sigma; Y) \\ & \ll
\br{(s-1)\zeta(s)\textstyle\prod_{p \mid q}\br{1 - \frac{1}{p^s}}}^{1/\phi(q)}
(q^{\frac{1}{2}}\log q)(1 + \ab{\tau})^{1 - 1/\phi(q)} \Pi_2(\sigma; Y).
\end{align*}
For $s - 1 \ll 1$, we can use \eqref{(4.2)} to obtain (recalling the definition \eqref{(4.20)} of $\Pi_1(\sigma, q)$):
\begin{align*}
 \br{(s-1)\zeta(s)\textstyle\prod_{p \mid q}\br{1 - \frac{1}{p^s}}}^{1/\phi(q)} \ll \Pi_1(\sigma; q).
\end{align*}
For $s \in \mathcal{D}^{*}$ with $\ab{\tau} \ge 2$, we use \eqref{(4.15)} to obtain
\begin{align*}
 \br{(s-1)\zeta(s)\textstyle\prod_{p \mid q}\br{1 - \frac{1}{p^s}}}^{1/\phi(q)} \ll (\log\ab{\tau})^{1/\phi(q)}\Pi_1(\sigma; q).
\end{align*}
Combining gives \eqref{(4.23)} (upon noting that $1 + \ab{\tau} \ll 1$ if $s - 1 \ll 1$), and \eqref{(4.24)}.

We may express $G(s)$ as a Taylor series in a neighbourhood of $1$. Let $s \in \mathcal{D}^{*}$ with $s - 1 \ll 1$ be given, and choose a positive number $r \ll 1$ so that $r - \ab{s-1} \gg 1$, and so that the circle $\gamma(r) \defeq  \{1 + re^{i\theta} : -\pi < \theta < \pi\}$, without the point $1 - r$, is contained in $\mathcal{D}^{*}$. By the Cauchy integral formulae, we have
\begin{align*}
 G(s) & = G(1) + \sum_{j=1}^{\infty} \frac{G^{(j)}(1)}{j!}(s-1)^j \\ & = 
G(1) + \sum_{j=1}^{\infty} \frac{(s-1)^j}{2\pi i} \int_{\gamma(r)} \frac{G(w)}{(w-1)^{j+1}} \, \dd{w} \\ & = 
G(1) + \BigO{\max_{\ab{w-1} = r} \ab{G(w)} \sum_{j=1}^{\infty} \frac{  {\phantom{^j}}  \ab{s-1}^j}{r^j}     }.
\end{align*}
 By our choice of $r$, we have
\begin{align*}
 \sum_{j=1}^{\infty} \frac{ {\phantom{^j}} \ab{s-1}^j}{r^j} = \frac{\ab{s-1}}{r - \ab{s-1}} \ll \ab{s-1}.
\end{align*}
Using \eqref{(4.24)} for $\max_{\ab{w-1} = r} \ab{G(w)}$ and combining gives \eqref{(4.25)}.

(b) By definition \eqref{(4.19)} of $F(s)$ and analyticity of $G(s)$ throughout $\mathcal{D}^{*}$, we conclude that $F(s)$ is also analytic throughout $\mathcal{D}^{*}$. It has a branch point at $s = 1$. Now, for $\sigma > 1$, we can write 
\begin{align*}
 F(s) = \Theta(s) 
\br{\prod_{\chi \bmod q} L(s,\chi)}^{1/\phi(q)}
\prods{p \le Y}{p \equiv 1 \bmod q} \br{1 - \frac{1}{p^s}},
\end{align*}
and as in \eqref{(4.3)} --- \eqref{(4.5)}, we see that, for $\sigma > 1$,
\begin{align*}
\log \Theta(s) + \frac{1}{\phi(q)}\log \br{\textstyle\prod_{\chi \bmod q} L(s,\chi)} = 
\sum_{p \equiv 1 \bmod q} \log\br{1 - \frac{1}{p^s}}^{-1}.
\end{align*}
Exponentiating and combining gives \eqref{(4.26)}.
\end{proof}

We are finally ready to prove Lemma \ref{Lemma 3.3}. The proof is an adaptation of the Selberg-Delange method, as presented in \cite[\S II.5]{T1995}.

\begin{proof}[Proof of Lemma 3.3]
Fix numbers $A > 0$ and $\alpha \in (0,\frac{1}{2})$. Let $X$ be a sufficiently large number, fix an integer $q$ satisfying $3 \le q \le (\log X)^{\alpha}$, and fix a number $Y$ satisfying $1 \le Y \le (\log X)^A$. Let
\begin{align*}
a_n \defeq
\begin{cases}
1 & \text{of $p \mid n \Rightarrow p \equiv 1 \bmod q$ and $p > Y$,} \\
0 & \text{otherwise,}
\end{cases}
\end{align*}
be the characteristic function of the integers composed only of primes that are both congruent to $1$ mod $q$ and greater than $Y$, and let
\begin{align*}
F(s) \defeq \sum_{n=1}^{\infty} \frac{a_n}{n^s} = 
\prod_{p \equiv 1 \bmod q} \br{1 - \frac{1}{p^s}}^{-1} 
\prods{p \le Y}{p \equiv 1 \bmod q} \br{1 - \frac{1}{p^s}}, \quad \sigma > 1,
\end{align*}
be its associated Dirichlet series and Euler product. Let $\kappa$ and $T$ be parameters, to be determined later, but satisfying $1 < \kappa \le 1 + 1/\log X$ and $T > 1$. By equation \eqref{(4.13)} of Theorem \ref{Theorem 4.3}, we have
\begin{align}\label{(4.28)}
\sumss{n \le X}{p \mid n \Rightarrow p \equiv 1 \bmod q}{\text{and $p > Y$}}\hspace{-5pt} 1 =
\sum_{n \le X} a_n = \frac{1}{2\pi i} \int_{\kappa - iT}^{\kappa + iT} F(s) \frac{ {\phantom{^s}} X^s}{s} \, \dd{s} + \BigO{1} + \BigO{(X\log X)/T}.
\end{align}

By Lemma \ref{Lemma 4.5} (b), $F(s)$ admits an analytic continuation to the simply connected domain $\mathcal{D}^{*}$ (defined in \eqref{(4.17)}), namely \eqref{(4.19)}:
\begin{align*}
F(s) & = (s-1)^{-1/\phi(q)}G(s) \\ & = 
\Theta(s)\br{\textstyle\zeta(s)\prod_{p \mid q}\br{1 - \frac{1}{p^s}}\prod_{\chi \ne \chi_0} L(s,\chi)}^{1/\phi(q)}
\prods{p \le Y}{p \equiv 1 \bmod q}\br{1 - \frac{1}{p^s}}, \quad s \in \mathcal{D}^{*}.
\end{align*}
Here, $G(s)$ is as in \eqref{(4.18)} (analytic throughout $\mathcal{D}^{*} \cup \{1\}$), and $\Theta(s)$ is as in \eqref{(4.6)} (analytic throughout the half-plane $\sigma > 1/2$). The Cauchy integral theorem then allows us to deform the segment of integration $[\kappa - iT, \kappa + iT]$ in \eqref{(4.28)} into a closed, rectifiable path $\mathscr{C}$, joining its end-points and lying inside $\mathcal{D}^{*}$. 

Let us define our contour $\mathscr{C}$. Let $\kappa$, $T$, and $\eta$ be such that the rectangle with corners $(1 - \eta, \pm iT)$ and $(\kappa, \pm iT)$, with the point $1 - \eta$ deleted, lies inside the simply connected domain $\mathcal{D}^{*}$. The contour $\mathscr{C}$ is this same rectangle, with the point $1 - \eta$ deleted, traversed clockwise, and with a detour taken around $s = 1$ via the truncated `Hankel' contour, $\mathscr{H}$. The contour $\mathscr{H}$ consists of the circle $\ab{s - 1} = r$ (the only condition on $r$ is that $0 < r < \kappa - 1$), excluding the point $s = 1-r$, together with the line segment $[1-\eta, 1 - r]$, traced out twice, with respective arguments $-\pi$ and $+\pi$. That is, 
\begin{align*}
\mathscr{H} \defeq [1 + \eta e^{-\pi i}, 1 + re^{-\pi i}] \cup 
\{1 + r e^{i\theta}: -\pi < \theta < \pi \} \cup 
[1 + r e^{\pi i}, 1 + \eta e^{\pi i}].
\end{align*}
We denote the left vertical line segments of $\mathscr{C}$ by $C_1$ and the horizontal line segments by $C_2$. That is, 
\begin{align*}
 C_1 & \defeq [1 - \eta - iT, 1 - \eta) \cup (1 - \eta, 1 - \eta + iT]; \\
 C_2 & \defeq [\kappa - iT, 1 - \eta - iT] \cup [1 - \eta + iT, \kappa + iT].
\end{align*}
Thus,
\begin{align*}
 \mathscr{C} = C_1 \cup C_1 \cup \mathscr{H} \cup [\kappa + iT, \kappa -iT].
\end{align*}

By Cauchy's integral theorem we have
\begin{align*}
\Br{ \int_{C_1+C_2+\mathscr{H}} - 
\int_{\kappa - iT}^{\kappa + iT} }
 F(s) \frac{{\phantom{^{s}}} X^s}{s} \, \dd{s} = 
\int_{\mathscr{C}} F(s) \frac{{\phantom{^{s}}} X^s}{s} \, \dd{s} = 0, 
\end{align*}
and so
\begin{align}\label{(4.29)}
 \frac{1}{2\pi i} \int_{\kappa - iT}^{\kappa + iT}  F(s) \frac{{\phantom{^{s}}} X^s}{s} \, \dd{s} =  
\frac{1}{2\pi i} \int_{\mathscr{H}} F(s) \frac{{\phantom{^{s}}} X^s}{s} \, \dd{s} + 
\BigO{\int_{C_1 + C_2} F(s) \frac{{\phantom{^{s}}} X^s}{s} \, \dd{s}}.
\end{align}
Let us return to the $\Oh$-term later. We will now show that
\begin{align}\label{(4.30)}
\frac{1}{2\pi i} \int_{\mathscr{H}} F(s) \frac{{\phantom{^{s}}} X^s}{s} \, \dd{s}  =
\frac{X(\log X)^{\frac{1}{\phi(q)}}}{\log X} \br{ \frac{ G(1) } { \Gamma(  1/\phi(q)  ) } + 
\BigO{\mathscr{E}_1} + \BigO{\mathscr{E}_2}},
\end{align}
where
\begin{align*}
\mathscr{E}_1 \defeq  \frac{(q^{\frac{1}{2}}\log q) \Pi_1(1 - \eta;q)\Pi_2(1 - \eta;Y)}{\log X}, \quad
\mathscr{E}_2 \defeq G(1)X^{-\eta}.
\end{align*}

On the Hankel contour $\mathscr{H}$ we have $s \asymp 1$, that is $1/s = 1 - (s-1)/s = 1 + \BigO{\ab{s-1}}$. Therefore, by \eqref{(4.24)} and \eqref{(4.25)} of Lemma \ref{Lemma 4.5} (a), for $s \in \mathscr{H}$, we have
\begin{align*}
 \frac{G(s)}{s} = G(s) + \BigO{\ab{s-1}G(s)} = G(1) + \BigO{ \ab{s-1} (q^{\frac{1}{2}}\log q) \Pi_1(\sigma;q) \Pi_2(\sigma;Y)}.
\end{align*}
Hence, since $F(s) = (s-1)^{-1/\phi(q)}G(s)$, and since $\Pi_{1}(\sigma;q)$ and $\Pi_{2}(\sigma;Y)$ are bounded respectively by $\Pi_{1}(1 - \eta;q)$ and $\Pi_{2}(1 - \eta;Y)$ on $\mathscr{H}$ (recall the definitions \eqref{(4.20)}), we have
\begin{align}\label{(4.31)}
\begin{split}
& \frac{1}{2\pi i} \int_{\mathscr{H}} F(s)\frac{{\phantom{^s}}X^s}{s} \, \dd{s}   = 
\frac{G(1)}{2\pi i} \int_{\mathscr{H}} \frac{X^s}{(s-1)^{1/\phi(q)}} \, \dd{s} \\ & \hspace{50pt} + 
\BigO{ (q^{\frac{1}{2}}\log q) \Pi_1(1 - \eta;q) \Pi_2(1 - \eta;Y) \int_{\mathscr{H}} X^{\sigma} \ab{s-1}^{1 - 1/\phi(q)} \, \dd{s}}.
\end{split}
\end{align}

Via the substitution $\sigma = 1 - s$, we see that the straight line segments of $\mathscr{H}$ contribute
\begin{align*}
& \frac{G(1)}{2\pi i} \br{ -
\int_{\eta}^{r} \frac{X^{1-\sigma}}{(\sigma e^{-i\pi})^{1/\phi(q)}} \, \dd{\sigma} + 
\int_{r}^{\eta} \frac{X^{1-\sigma}}{(\sigma e^{i\pi})^{1/\phi(q)}} \, \dd{\sigma}  } \\ & \hspace{100pt} = 
\frac{G(1)\sin(\pi/\phi(q))}{\pi}  \int_{r}^{\eta} X^{1-\sigma}\sigma^{-1/\phi(q)} \, \dd{\sigma}
\end{align*}
to the main term on the right-hand side of \eqref{(4.31)}. The circle contributes
\begin{align*}
\frac{G(1)}{2\pi i} 
\int_{-\pi}^{\pi} \frac{X^{1 + re^{i\theta}} } {(re^{i\theta})^{1/\phi(q)}} \, ire^{i\theta} \dd{\theta}  
\le G(1) X^{1+r} r^{1 - 1/\phi(q)}  
\ll G(1)Xr^{1 - 1/\phi(q)},
\end{align*}
because $r < \kappa - 1 < 1/\log X$. As all of this holds for arbitrarily small $r$, we conclude that 
\begin{align*}
\frac{G(1)}{2\pi i} \int_{\mathscr{H}} \frac{X^s}{(s-1)^{1/\phi(q)}} \, \dd{s} & =
\frac{G(1)\sin(\pi/\phi(q))}{\pi}   \int_{0^+}^{\eta} X^{1-\sigma}\sigma^{-1/\phi(q)} \, \dd{\sigma} \\ & =
\frac{G(1)\sin(\pi/\phi(q))}{\pi} \cdot \frac{X(\log X)^{ \frac{1}{\phi(q)} } }{\log X}
  \int_{0^+}^{\eta\log X} e^{-u}u^{1/\phi(q)} \, \dd{u},
\end{align*}
after the substitution $u = \sigma\log X$. We can approximate the last integral by $\Gamma(1 - 1/\phi(q))$:
\begin{align*}
\int_{0^+}^{\eta\log X} e^{-u}u^{-1/\phi(q)} \, \dd{u} & = 
\Br{\int_{0^+}^{\infty} - \int_{\eta\log X}^{\infty}} e^{-u}u^{-1/\phi(q)} \, \dd{u} \\ & =
\br{\textstyle\Gamma\br{1 - \frac{1}{\phi(q)}} + \BigO{X^{-\eta}}}.
\end{align*}
Combining and using the identity $\Gamma(\theta)\Gamma(1 - \theta) = \pi/\sin(\pi\theta)$, we obtain
\begin{align}\label{(4.32)}
 \frac{G(1)}{2\pi i} \int_{\mathscr{H}} \frac{X^s}{(s-1)^{1/\phi(q)}} \, \dd{s}  =
\frac{X(\log X)^{ \frac{1}{\phi(q)} } }{\log X} \br{\frac{G(1)}{\Gamma(1/\phi(q))} + \BigO{G(1)X^{-\eta}}}.
\end{align}

Similarly, to the integral over $\mathscr{H}$ in the $\Oh$-term of \eqref{(4.31)}, each straight line segment contributes
\begin{align*}
\int_{1 - \eta}^{1 - r} X^{\sigma} \ab{s-1}^{1 - 1/\phi(q)} \dd{s} & \le 
\int_{1 - \eta}^{1} X^{\sigma}(1 - \sigma)^{1 - 1/\phi(q)} \dd{\sigma} \\ & =
\frac{X(\log X)^{ \frac{1}{\phi(q)} } }{(\log X)^2} \int_{0}^{\eta\log X} e^{-u} u^{1 - 1/\phi(q)} \dd{u} \\ & = 
\frac{X(\log X)^{ \frac{1}{\phi(q)} } }{(\log X)^2} \br{ \Gamma(2 - 1/\phi(q)) + X^{-\eta}\eta^{1 - 1/\phi(q)}  }. 
\end{align*}
(We used the substitution $u = (1 - \sigma)\log X$ and approximated the resulting integral by $\Gamma(2 - 1/\phi(q))$.) 
To the integral over $\mathscr{H}$ in the $\Oh$-term of \eqref{(4.31)}, the circle $\ab{s-1} = r$ contributes at most
\begin{align*}
 X^{1 + \sigma - 1} \int_{\ab{s-1} = r} r^{1 - 1/\phi(q)} \ab{\dd{s}} \le 2\pi X^{1+r} r^{2 - 1/\phi(q)} \ll X r^{2 - 1/\phi(q)}.
\end{align*}
Letting $r$ tend to zero and combining gives
\begin{align}\label{(4.33)}
 \int_{\mathscr{H}} X^{\sigma} \ab{s-1}^{1 - 1/\phi(q)} \, \dd{s} \ll 
\frac{X(\log X)^{ \frac{1}{\phi(q)} } }{\log X}\cdot \frac{1}{\log X}.
\end{align}
Putting \eqref{(4.33)} and \eqref{(4.32)} into \eqref{(4.31)} gives \eqref{(4.30)}.

Let us turn now to the integral over $C_1 + C_2$ in the $\Oh$-term of \eqref{(4.29)}. First note that $\Pi_1(\sigma;q)$ and $\Pi_2(\sigma;Y)$ (defined in \eqref{(4.20)}), attain their maximums on $\mathscr{C}$ when $\sigma = 1 - \eta$. On $C_1$, we use \eqref{(4.23)} to bound the integrand for $\ab{\tau} \ge 2$, and \eqref{(4.24)}, plus the fact that $\ab{s-1} \ge \eta$ on $C_1$, to bound the integrand for $\ab{\tau} \le 2$ (that is $s - 1 \ll 1$). Thus, 
\begin{align*}
F(s) = (s-1)^{-1/\phi(q)}G(s) \ll 
\begin{cases}
T(\log T) (q^{\frac{1}{2}}\log q)  \Pi_1(1 - \eta;q)\Pi_2(1 - \eta;Y) & \text{if $\ab{\tau} \ge 2$,} \\
\eta^{-1} (q^{\frac{1}{2}}\log q)  \Pi_1(1 - \eta;q)\Pi_2(1 - \eta;Y) & \text{if $\ab{\tau} \le 2$.}
\end{cases}
\end{align*}
We see that
\begin{align}\label{(4.34)}
\begin{split}
 & \int_{C_1} F(s) \frac{{\phantom{^{s}}} X^s}{s} \, \dd{s} \\ & \hspace{30pt} \ll 
X^{1 - \eta} \br{  \eta^{-1} + T\log T \int_{2}^{T}  \frac{\dd{\tau}}{\tau} }  
(q^{\frac{1}{2}}\log q) \Pi_1(1 - \eta;q)\Pi_2(1 - \eta;Y) \\ & \hspace{30pt} \ll
X^{1 - \eta}\br{\eta^{-1} + T(\log T)^2} (q^{\frac{1}{2}}\log q)\Pi_1(1 - \eta;q)\Pi_2(1 - \eta;Y).
\end{split}
\end{align}
On $C_2$ we have $\ab{s} \asymp T$, so we just use \eqref{(4.23)} and the fact that $\ab{X^s} \le X^{\kappa} \ll X$ ($\kappa \le 1 + 1/\log X$):
\begin{align}\label{(4.35)}
\begin{split}
& \int_{C_2} F(s) \frac{{\phantom{^{s}}} X^s}{s} \, \dd{s} \\ & \hspace{30pt} \ll
(\eta + \kappa) \frac{X}{T} (\log T)^{1/\phi(q)} T^{1 - 1/\phi(q)} 
(q^{\frac{1}{2}}\log q) \Pi_1(1 - \eta;q)\Pi_2(1 - \eta;Y) \\ & \hspace{30pt} \ll
(\eta + \kappa - 1) X \br{\frac{\log T}{T}}^{1/\phi(q)} (q^{\frac{1}{2}}\log q) \Pi_1(1 - \eta;q)\Pi_2(1 - \eta;Y).
\end{split}
\end{align}

We are now ready to choose our parameters $\kappa$, $T$, and $\eta$. With $\bar{c}$ as in Theorem \ref{Theorem 4.4}, and as in the definition \eqref{(4.17)} of $\mathcal{D}^{*}$, we set
\begin{align*}
\kappa \defeq 1 + \frac{1}{\log X}, \quad T \defeq \exp\br{\textstyle\frac{1}{4}(\bar{c}\log X)^{1/2}}, \quad
\eta \defeq \frac{\bar{c}}{2\log T} = \frac{2\bar{c}}{(\bar{c}\log X)^{1/2}}.
\end{align*}
Straightforward calculations reveal that, for all sufficiently large $X$,
\begin{align}\label{(4.36)}
\begin{split}
 X^{1 - \eta}\br{\eta^{-1} + T(\log T)^2} & = 
\exp\br{-\textstyle\frac{7}{4}(\bar{c}\log X)^{1/2} + \log\log X + \BigO{1}} \\ & \le  
X\exp\br{-\textstyle\frac{3}{2}(\bar{c}\log X)^{1/2}},
\end{split}
\end{align}
 and
\begin{align}\label{(4.37)}
\begin{split}
 (\eta + \kappa - 1) X \br{\frac{\log T}{T}}^{1/\phi(q)} & \ll 
\eta X \exp\br{ - \frac{(\bar{c} \log X)^{1/2}}{4\phi(q)} + \frac{\log\log X}{\phi(q)} + \BigO{1} } \\ & \ll
 X \exp\br{ - \frac{(\bar{c} \log X)^{1/2}}{5\phi(q)} }.
\end{split}
\end{align}

Putting \eqref{(4.37)} into \eqref{(4.35)}, and \eqref{(4.36)} into \eqref{(4.34)}, then putting the resulting bounds, as well as \eqref{(4.30)}, into \eqref{(4.29)}, we obtain
\begin{align}\label{(4.38)}
\begin{split}
 & \frac{1}{2\pi i} \int_{\mathscr{H}} F(s) \frac{{\phantom{^{s}}} X^s}{s} \, \dd{s} \\ & \hspace{50pt} =
\frac{X(\log X)^{\frac{1}{\phi(q)}}}{\log X} \br{\frac{G(1)}{\Gamma(1/\phi(q))} +
\BigO{ \mathscr{E}_1} + \BigO{  \mathscr{E}_2 } + \BigO{  \mathscr{E}_3} + \BigO{ \mathscr{E}_4}},
\end{split}
\end{align}
where 
\begin{align*}
 \mathscr{E}_1 & = \frac{ (q^{\frac{1}{2}}\log q)  \Pi_1(1 - \eta; q)\Pi_2(1 - \eta;q)}{\log X}; \\
\mathscr{E}_2 & = G(1)X^{-\eta} \ll (q^{\frac{1}{2}}\log q)  \Pi_1(1 - \eta; q)\Pi_2(1 - \eta;q); \\
\mathscr{E}_3 & = (\log X)^{1 - 1/\phi(q)} (q^{\frac{1}{2}}\log q)  \Pi_1(1 - \eta; q)\Pi_2(1 - \eta;q)
\exp\br{-\textstyle\frac{3}{2}(\bar{c}\log X)^{1/2}} \\ & \ll
(q^{\frac{1}{2}}\log q) \Pi_1(1 - \eta; q)\Pi_2(1 - \eta;q)
\exp\br{-(\bar{c}\log X)^{1/2}}; \\
\mathscr{E}_4 & = (\log X)^{1 - 1/\phi(q)}(q^{\frac{1}{2}}\log q) \Pi_1(1 - \eta; q)\Pi_2(1 - \eta;q) 
\exp\br{ - \frac{(\bar{c} \log X)^{1/2}}{5\phi(q)} }; \\ & \ll
(q^{\frac{1}{2}}\log q) \Pi_1(1 - \eta; q)\Pi_2(1 - \eta;q) 
\exp\br{ - \frac{(\bar{c} \log X)^{1/2}}{6\phi(q)} }.
\end{align*}
(We used \eqref{(4.24)} to bound $\mathscr{E}_2$.)

Now $\phi(q)$ may be of the same order as $q$, for instance if $q$ is prime, but as $q \le (\log X)^{\alpha}$ for some $\alpha \in (0,\frac{1}{2})$, the largest error term here, for all sufficiently large $X$, is $\BigO{\mathscr{E}_1}$. Thus
\begin{align}\label{(4.39)}
 \BigO{\mathscr{E}_1} + \BigO{\mathscr{E}_2} + \BigO{\mathscr{E}_3} + \BigO{\mathscr{E}_4} =
  \BigO{\mathscr{E}_1} = \BigO{\frac{q^{\frac{1}{2}}(\log\log X)^c}{\log X}}.
\end{align}
This bound holds uniformly for $1 \le q \le (\log X)^{\alpha}$ and $1 \le Y \le (\log\log X)^A$ ($A > 0$ given), by \eqref{(4.21)} and \eqref{(4.22)}.

Note that
\begin{align}\label{(4.40)}
G(1) = c(q)\prods{p \le Y}{p \equiv 1 \bmod q} \br{1 - \frac{1}{p}},
\end{align}
where 
\begin{align*}
c(q) \defeq \Theta(1)\br{\frac{\phi(q)}{q} \prod_{\chi \bmod q} L(1,\chi)}^{1/\phi(q)}
\end{align*}
is the constant of Lemma \ref{Lemma 3.2}, defined in \eqref{(4.12)}. We have $\Theta(1) \ll 1$ by \eqref{(4.27)}; we have $\Gamma(1/\phi(q)) \ll \phi(q)$; we also have the lower bounds for $L(1,\chi)$ in \eqref{(4.16)}; and so
\begin{align*}
\frac{q^{\frac{1}{2}}\Gamma(1/\phi(q))}{c(q)} & \ll 
q^{\frac{1}{2}} \phi(q)\br{\frac{\phi(q)}{q} \cdot q^{\frac{1}{2}(\phi(q)-1)} }^{1/\phi(q)} \ll
\phi(q)^{1 + 1/\phi(q)} q^{1 - 1/\phi(q)} \ll
q^{2}.
\end{align*}
Since, for $Y \le (\log X)^A$, we also have
\begin{align*}
\prod_{p \le Y} \br{1 - \frac{1}{p}}^{-1} \ll \log Y \le A\log\log X \ll (\log\log X)^c,
\end{align*}
combining \eqref{(4.40)} and \eqref{(4.39)} with \eqref{(4.38)} gives
\begin{align}\label{(4.41)}
\begin{split}
 & \frac{1}{2\pi i} \int_{\mathscr{H}} F(s) \frac{{\phantom{^{s}}} X^s}{s} \, \dd{s}  \\ & \hspace{30pt} =
 \br{\frac{c(q)}{\Gamma(1/\phi(q))} + \BigO{\frac{q^{\frac{1}{2}}(\log\log X)^c}{\log X}} }
 \frac{X(\log X)^{\frac{1}{\phi(q)}}}{\log X}  \hspace{-5pt} 
 \prods{p \le Y}{p \equiv 1 \bmod q} \hspace{-5pt} \br{1 - \frac{1}{p}} \\ & \hspace{30pt} =
 \br{1 + \BigO{\frac{q^{2}(\log\log X)^c}{\log X}} } 
 \frac{c(q)}{\Gamma(1/\phi(q))} \cdot 
 \frac{X(\log X)^{\frac{1}{\phi(q)}}}{\log X}  \hspace{-5pt} \prods{p \le Y}{p \equiv 1 \bmod q} \hspace{-5pt} \br{1 - \frac{1}{p}}.
 \end{split}
\end{align}

Putting \eqref{(4.41)} into \eqref{(4.28)}, where the $\Oh$-term on the right-hand side is 
\begin{align*}
\BigO{1} + \BigO{(X\log X)/T} = \BigO{Xe^{-c\sqrt{\log X}}},
\end{align*}
we obtain
\begin{align*}
\sumss{n \le X}{p \mid n \Rightarrow p \equiv 1 \bmod q}{\text{and $p > Y$}}\hspace{-5pt} 1 & =
\br{1 + \BigO{\frac{q^{2}(\log\log X)^c}{\log X}} } \\ & \hspace{100pt} \times
 \frac{c(q)}{\Gamma(1/\phi(q))} \cdot 
 \frac{X(\log X)^{\frac{1}{\phi(q)}}}{\log X}  \hspace{-5pt} \prods{p \le Y}{p \equiv 1 \bmod q} \hspace{-5pt} \br{1 - \frac{1}{p}}.
\end{align*}
Noting that $q^2/\log X \le (\log X)^{2\alpha}/\log X$, Lemma \ref{Lemma 3.3} is proved.
\end{proof}

\section{Acknowledgements}\label{Acknowledgements}

I am grateful to Professor G\'erald Tenenbaum for pointing out that Lemma \ref{Lemma 3.3} is amenable to the Selberg-Delange method, and for providing references; to Pankaj Vishe for helpful discussions regarding some of the details of the proof; and to Daniel Fiorilli for reading and commenting on this manuscript.

\bibliographystyle{article}

\end{document}